\documentclass[12pt]{elsarticle}% 无预投稿水印
\usepackage{amsmath}
\usepackage{amssymb}
\usepackage{amsthm}
\usepackage{verbatim}
\usepackage{subfigure}
\usepackage{latexsym}
\usepackage{lmodern}
\usepackage{array}
\usepackage{graphicx}
\usepackage{float}
\usepackage{multirow}
\usepackage{listings}
\usepackage{color}
\usepackage{bm}
\usepackage{url}
\usepackage[english]{babel}
\usepackage{hyperref}
\usepackage{algorithm}
\usepackage{algorithmic}
\usepackage{multicol,balance}
\usepackage{times}
\usepackage{cases}
\usepackage{titlesec}
\usepackage[title]{appendix}
\usepackage{mdwlist}
\usepackage{makecell}
\usepackage{booktabs}
\usepackage{mathabx}
\usepackage{epstopdf}
\usepackage{framed}
\usepackage{tikz}
\allowdisplaybreaks[3]

\newcommand{\R}{\mathbb{R}}

\newcommand*\xbar[1]{%
	\hbox{%
		\vbox{%
			\hrule height 0.5pt % The actual bar
			\kern0.5ex%         % Distance between bar and symbol
			\hbox{%
				\kern-0.1em%      % Shortening on the left side
				\ensuremath{#1}%
				\kern-0.1em%      % Shortening on the right side
			}%
		}%
	}%
}

\newtheorem{lemma}{Lemma}[section]

\newtheorem{prop}{Proposition}[section]

\newtheorem{theorem}{Theorem}[section]

\newtheorem{definition}{Definition}[section]

\numberwithin{equation}{section}

\normalsize \makeatletter

\newcommand*\bbar[1]{%
	\hbox{%
		\vbox{%
			\hrule height 0.5pt % The actual bar
			\kern0.365ex%         % Distance between bar and symbol
			\hbox{%
				\kern-0.1em%      % Shortening on the left side
				\ensuremath{#1}%
				\kern-0.1em%      % Shortening on the right side
			}%
		}%
	}%
} 

\begin{document}

\begin{frontmatter}
\title
{
	Koch-Tataru theorem for 3D incompressible active nematic liquid crystals
%	\footnote{
%		The authors are partially supported by NSFC-11831003, NSFC-12031012 and the Institute of Modern Analysis-A Frontier Research Center of Shanghai.}
}
%%
%%
%\author[rvt1]{Quoc-Hung Nguyen}
%\ead{qhnguyen@amss.ac.cn}

\author[rvt2]{Fan Yang}
\ead{fanyang-m@ntu.edu.cn}

%\address[rvt1]{Academy of Mathematics and Systems Science, Chinese Academy of Sciences, Beijing, 100190, China}

\address[rvt2]{School of Mathematics and Statistics, Nantong University, Nantong, 226019, China}

\begin{abstract}
We investigate the incompressible hydrodynamic system of the active nematic liquid crystals in the Beris-Edwards framework. Although we focus on constant activity in this paper, the simplified system derived from it exhibits the potential to perform computations and transmit information in active soft materials \cite{defect-active}.
More precisely, by employing Kato's strategy for constructing mild solutions combined with the Banach contraction principle, we show the existence and uniqueness of the Koch-Tataru type solution in $\R^3$ for small initial data $(Q_0,u_0)\in L^\infty\times {\rm BMO}^{-1}$. This is the first well-posedness result for the system with initial data in critical space.
%Our results contribute to a deeper understanding of 

\end{abstract}

\end{frontmatter}
%
%\linenumbers
%\newpage

%\tableofcontents
%
%\newpage

\section{Introduction}
In physics, active matter describes natural or artificial systems that are out of thermodynamic equilibrium through the transduction of energy derived from an internal source or ambient medium.
Typical examples are spanning from macroscopic (e.g. shoals of fish and flocks of birds \cite{exp-fish}) to microscopic scales (e.g. actin filaments \cite{exp-actin}, bacteria \cite{exp-bacteral}, cytoskeletal filaments \cite{exp-cyto} and microtubule bundles \cite{exp-microtubules}).
In general, such systems consist of individual units, referred to as active particles or self-driven particles, which may interact both directly as well as through disturbances propagated via the surrounding environment.
The most interesting feature of active fluids is the emergence of collective motion, such as non-equilibrium phase transitions between novel dynamical phases, active turbulence and the associated motile topological defects.
Indeed, there is increasing evidence that such collective behaviour is important in developing new technologies, for instance, designing autonomous materials systems endowed with logic operations \cite{defect-active}, and developing optimized flocking model for real drones in confined spaces \cite{drone} and biohybrid microrobots for targeted delivery of therapeutics \cite{microrobot}.
Therefore, understanding the laws underlying the collective dynamics of active matter has important implications for natural systems across a wide range of length scales.

In the last decades, many theoretical and experimental models of active matter have been developed to study their collective behaviour.
Since active particles are usually elongated, and their self-propulsion direction is influenced by their own anisotropy instead of being dictated by an external field.
To capture this feature, a generic approach is to utilize the concept of orientation fields from liquid crystals \cite{model1,model2}. Although there are two basic types of orientational order for rod-like particles, we are more interested in a nematic order that described by a traceless symmetric tensor field (conversely, the case of active particles with polar order is characterized by a vector field $\bm{p}$; see \cite{polar-case}).
Moreover, this type of model for active matter is often called {\em active nematic liquid crystals}.
For more information and discussions, we refer the interested readers to \cite{blow2014,active-review,Giomi2011,Giomi2012,Giomi2013,Giomi2014,activematter,Ravnik2013} and the references therein.

More precisely,
we are concerned with the well-known dynamical model for the incompressible flow of active nematic fluids that proposed in \cite{active-servy,Giomi2012}, which is a system that couples a forced incompressible Navier-Stokes equation for the underlying fluid velocity field with equations of motion for the orientation field and the concentration of active particles.
That is
\begin{equation}\label{eqiacl}
	\begin{cases}
		\partial_t c + u\cdot\nabla c = \nabla\cdot \left[ \left( D_0 \mathbb{I}_3+ D_1Q \right)\nabla c + \alpha_1 c^2\nabla\cdot Q  \right],\\
		\partial_t Q + u\cdot\nabla Q  = S(\nabla u,Q) +\lambda |Q|D+\Gamma H[Q,c],\\
		\partial_t u + u\cdot\nabla u = \mu\Delta u-\nabla p + \nabla\cdot \left( \tau+\sigma \right),\\
		\nabla\cdot u = 0, \hspace{4.4cm} x\in \R^3,\quad t>0,
	\end{cases}
\end{equation}
where the scalar active concentration $c>0$ is considered to vary with time and space, reflecting the effect caused by activity and flow.
In addition, $u \in \R^3$ is the flow velocity and $p$ is the usual pressure.
Based on the Landau-de Gennes theory, the nematic tensor order parameter $Q$ is a traceless symmetric $3\times 3$ matrix, and describes the primary and secondary directions of nematic alignment along with variations in the degree of nematic order \cite{Q-tensor-book}.
For simplicity, we generally refer to it as "$Q$-tensor".
It is also important to note that, alongside the passive liquid crystals, $Q$ vanishes in the isotropic phase; whereas the uniaxial phase corresponds to that $Q$ has two equal non-zero eigenvalues, and the biaxial phase provides that $Q$ has three distinct eigenvalues.
Meanwhile, the positive constants $D_0$ and $D_1$ are both diffusion coefficients, $\mu>0$ represents the viscosity coefficient, $\Gamma^{-1}>0$ is related to the rotational viscosity, $\lambda\in\R$ is the nematic alignment parameter, $\alpha_1$ refers to a constant with dimensions of inverse time, and $\mathbb{I}_3$ denotes the $3\times 3$ identity matrix.
The molecular field $H[Q,c]$ is defined as
\[
H=-\frac{\delta \mathcal{F}}{\delta Q}+ \frac{1}{3}\mathbb{I}_3{\rm Tr}\left(\frac{\delta \mathcal{F}}{\delta Q}\right),
\]
which drives the system toward thermodynamic equilibrium with the Landau-de Gennes free energy $\mathcal{F}$ (see also \cite{Giomi2012,Q-tensor-book}). In three dimensions, this reads
\[
\mathcal{F}=\int \left( \dfrac{K}{2}|\nabla Q|^2  + \dfrac{K}{4}(c-c_\star) {\rm Tr}(Q^2) - \dfrac{b}{3} {\rm Tr}(Q^3) + \dfrac{K}{4}c| {\rm Tr}(Q^2) |^2 \right) dA,
\]
where $K>0$ is the elastic constant, $b\in\R$ is a material-dependent constant, and $c_\star$ refers to the critical concentration for the IN (isotropic-nematic) transition. Therefore, $H[Q,c]$ can be written as
\begin{equation}
	H[Q,c]:= K\Delta Q -\dfrac{K}{2}(c-c_\star) Q  + b\left[ Q^2-\dfrac{{\rm Tr}(Q^2)}{3}\mathbb{I}_3\right]- Kc Q{\rm Tr}(Q^2).
\end{equation}
Moreover, the remaining stress terms in \eqref{eqiacl} are given by
\begin{align}
	S(\nabla u, Q):=&\xi D\left( Q+\dfrac{1}{3}\mathbb{I}_3 \right) + \xi \left( Q+\dfrac{1}{3}\mathbb{I}_3 \right)D - 2\xi \left( Q+\dfrac{1}{3}\mathbb{I}_3 \right){\rm Tr}(Q\nabla u) + \Omega Q- Q\Omega,\\
	\tau:=&-\xi \left( Q+\dfrac{1}{3}\mathbb{I}_3 \right)H - \xi H\left( Q+\dfrac{1}{3}\mathbb{I}_3 \right) + 2\xi \left( Q+\dfrac{1}{3}\mathbb{I}_3 \right){\rm Tr}(QH) - K\nabla Q \odot\nabla Q,\label{eq1.4}\\
	\sigma:=&\underbrace{-\lambda |Q|H+ QH-HQ}_{\text{elastic stress}} + \underbrace{\alpha_2 c^2Q}_{\text{active stress}}=-\lambda |Q|H+ Q\Delta Q-\Delta QQ + \alpha_2 c^2Q.\label{eq1.5}
\end{align}
Here $D(u):=\frac{1}{2}\left( \nabla u + \nabla u^{T} \right)$ and $\Omega(u):=\frac{1}{2}\left( \nabla u - \nabla u^{T} \right)$ denote the symmetric and antisymmetric part of the
%stress tensor with $(\nabla u)_{\alpha\beta}=\partial_\beta u_\alpha$
deformation tensor $\nabla u$, respectively.
% tensor with $(\nabla u)_{\alpha\beta}=\partial_\beta u_\alpha$,
The notation $\nabla Q \odot\nabla Q$ in \eqref{eq1.4} stands for a symmetric $3\times 3$ matrix with its component $(\nabla Q \odot\nabla Q)_{\alpha\beta}$ given by $\partial_\beta Q_{\gamma\delta}\partial_\alpha Q_{\gamma\delta}$.
A partial Einstein summation convention, as is typical, is used throughout the paper.
The term $S(\nabla u,Q)$ describes the contribution of rigid-body rotation
and shear alignment, and the parameter $\xi$ describes the ratio of tumbling and aligning effects on the molecules exerted by the flow.
Note that, such as a shear flow, liquid-crystalline materials do not rotate like a rigid body. Hence, $Q$ also tends to align with the streamlines of the flow.
%In this regard
Additionally,
$\xi=0$ corresponds to the corotational case that describes the molecules only tumble in a shear flow but do not align.
On the other hand, 
%In principle,
the tensors $\tau$ and $\sigma$ appearing here basically
%correspond to the symmetric and antisymmetric contribution of the stress tensor, respectively, and
take forms similar to those in passive liquid crystals, but $\sigma$ involves a new nonlinear term on the interaction of active particles.
%have a more complicated nonlinear structure.
In addition, the active stress defined as in \eqref{eq1.5} is due to the force dipoles, where activity parameter $\alpha_2$ is positive for contractile, negative for extensile and zero in equilibrium (see also \cite{activeterm2,activeterm1}).

Indeed, we can see from the first equation of system \eqref{eqiacl} that there exist strongly nonlinear couplings between $c$ and $Q$, which consequently pose significant analytical challenges.
To the best of our knowledge, there exists limited literature concerning the analytical study of this model. Besides the work in \cite{active-limit}, there are only a few research papers available that focus on the existing results of weak and strong solutions and their uniqueness properties for some simplified models under additional parameter restrictions; see for instance \cite{weak-ALC,ALC-decay,weak-strong-3d}. However, even for a simplified version of this model has been discussed in \cite{defect-active}, and exhibit the potential to perform computations and transmit information in active soft materials (e.g. actin-, tubulin-, and cell-derived systems). Therefore, it is worth developing a rigorous mathematical description and analysis for active nematic liquid crystals.

In this paper, we still restrict ourselves to the case of constant active concentration, and assume that
\begin{equation}
	K=1,\quad \xi=\alpha_1=0,\quad	a=\dfrac{1}{2}(c-c_\star)\quad {\rm and} \quad \kappa=\alpha_2c^2.
\end{equation}
To summarize, we arrive at a dynamical system for active nematic liquid crystals as follows (see \cite{weak-ALC}):
\begin{equation}\label{eq1.1}
	\begin{cases}
		\partial_tQ+(u\cdot\nabla)Q + Q\Omega-\Omega Q-\lambda |Q|D=\Gamma H[Q];\\
		\partial_t u + (u\cdot\nabla)u +\nabla p-\mu \Delta u =-\nabla\cdot (\nabla Q \odot\nabla Q)+\nabla\cdot (Q\Delta Q-\Delta QQ)\\
		\hspace{5.3cm}-\lambda\nabla\cdot(|Q|H[Q])+\kappa\nabla\cdot Q;\\
		\nabla\cdot u = 0,\hspace{4cm} x\in \R^3,\quad t>0,
	\end{cases}
\end{equation}
with
\begin{equation}\label{eq1.8}
	H[Q]:= \Delta Q-aQ+b\left[Q^2-\dfrac{{\rm Tr}(Q^2)}{3}\mathbb{I}_3\right]-cQ{\rm Tr}(Q^2),
\end{equation}
where $\mu>0$, $\Gamma>0$, $c>0$, $\lambda$, $a$, $b$, $\kappa\in \R$, and system \eqref{eq1.1} is supplemented with initial data
\begin{equation}
	(Q,u)(t,x)\mid_{t=0}=(Q_0,u_0)(x) \quad \text{\rm for}\quad x\in\R^3.
\end{equation}
Moreover, it's worth noting that this simplified system has proven effective in numerically analyzing the topological defects of active nematic fluids, providing the prospect of carrying out calculations and conveying data using active soft materials \cite{defect-active}.
%note that this simplified system has been successfully used to study the topological defects of active nematic liquid crystals numerically 

\subsection{Known results for active nematic liquid crystals}
As is well known, the dynamical systems of active liquid crystals at continuum level (in which the active agents are characterized by a smooth density field rather than the individual particles) are constructed by selectively adding phenomenological terms that represent activity to their passive continuum model; see \cite{polar-case,Giomi2011,Giomi2012,Giomi2013,activematter}.
In the case of constant active concentration studied here, system \eqref{eqiacl} can be reduced to the well-studied Beris-Edwards model (or the so-called $Q$-tensor system) by making certain assumptions. Hence, we begin with a short review of some recent results for these passive systems as a reference.

In \cite{Qtensor12}, Paicu and Zarnescu studied the corotational case of incompressible $Q$-tensor system and showed the existence of global weak solutions in $\R^d$ $(d=2,3)$, as well as the higher global regularity and weak-strong uniqueness in dimension two.
The same authors in \cite{Qtensor11} also extended the mentioned results to the non-corotational setting, where parameter $\xi$ is sufficiently small. 
Indeed, one can remove this smallness assumption in two dimensional periodic setting \cite{2dwithoutxi}. Besides, the uniqueness criterion of weak solutions in dimension three was obtained in \cite{weak-Q-tensor} for the case of $\xi=0$. Very recently, the author of this paper considered the non-corotational case ($\xi\ne0$) and showed that weak-strong uniqueness is still true under a new criterion; see \cite{weak-strong-LCD-xi} for the details. Xiao \cite{Q-unique-3d} adopted the maximal regularities of Stokes and parabolic operators to obtain the existence and uniqueness of strong solutions of the corotational system in a smooth bounded domain $\mathcal{O}\subset\R^3$.  Concerning a general $\xi\in\R$, in \cite{xi-all}, the authors also proved the global well-posedness of the corresponding system, through the method of operator theory and quasilinear equations.
Moreover, 
Feireisl, Rocca, Schimperna and Zarnescu \cite{nonisothermal1,nonisothermal2} addressed a nonisothermal variant of Beris-Edwards model in the case of Ball-Majumdar potential \cite{potential}, and showed the existence of global-in-time weak solutions for the system with periodic boundary conditions.
For the compressible $Q$-tensor system, the existence and long time behavior of global weak solutions were later studied in \cite{Q-tensor-comp}.
Nevertheless, it should be pointed out that many important progresses have been achieved for the modified $Q$-tensor systems in terms of the free energy or stress tensor; see \cite{ables2014,ables2016,2dlowSobolev,A-Z2016-2d,Du-wang-20,Weak-t-regularity,Hang-Ding-15,LC-decay-22,LC-decay-19,Q-tensor-Torus} and references therein.

For the active nematic liquid crystals, the first attempt was made by Chen, Majumdar, Wang and Zhang \cite{weak-ALC}. They proved the existence of global weak solutions for system \eqref{eq1.1} in dimension two and three, along with the regular solution for dimension two. But the related result for weak-strong uniqueness was solely obtained in $\R^2$ for the same reasoning as its passive counterpart.
Recently, in \cite{weak-strong-3d}, the author and Li addressed the weak-strong uniqueness of system \eqref{eq1.1} in three dimensional whole space. They particularly considered the so-called Leray-Hopf type weak solution which admits an energy inequality, and obtained a uniqueness criterion similar to that of Serrin for the incompressible Navier-Stokes equations \cite{serrin1963}. However, due to the strongly nonlinear coupling seen in \eqref{eq1.1}, it is still a challenging problem to construct such a weak solution for initial data $(Q_0,u_0)\in H^1\times L_\sigma^2$.
%because of the strongly nonlinear coupling in \eqref{eq1.1}.
Moreover, the existence of regular solutions ($H^s$-framework) and long-time behavior for \eqref{eq1.1} in $\R^3$ were studied in \cite{ALC-decay}. The author and Yang first proved the stable mixing estimates for the $Q$-tensor, which reveals that the threshold of exponential decay of the orientational field $Q$ is determined by the constant activity and related to the Landau-de Gennes free energy.
%, at least in the case when activity $c>c_\star$, active nematic phase becomes isotropic with an activity-dependent exponential convergence rate..
A stochastic analysis of the active hydrodynamics was also shown in \cite{stochastic-alc}.
In addition, Lian and Zhang \cite{Lian-zhang20} investigated the global weak solutions for the inhomogeneous version of system \eqref{eq1.1} in a smooth bounded domain. As for the compressible version, the authors in \cite{weak-ALC-comp}
%made some parameter assumptions and 
proved the existence of global weak solutions of the corresponding initial-boundary value problems. Indeed, one should note that these studies are far from what we expected for the active system that proposed in \cite{Giomi2012}.
To the best of our knowledge, only one very recent paper was dedicated to the existence theory of system \eqref{eqiacl}.
Jiang, Tang and Wang \cite{active-limit} derived the strong well-posedness for active system \eqref{eqiacl} in $\mathbb{T}^d$ with $d=2,3$, and the convergence problem of zero activity limit. For additional mathematical results, we refer the reader to \cite{active-servy,Jiang-com-s} and references therein.

\subsection{Motivation at the critical regularity}
%In order to offer
To provide a comprehensive and self-contained introduction to our work, we first recall some related results of the following well-known incompressible Navier–Stokes equations:
\begin{equation}\label{eq-ns}
	\begin{cases}
		\partial_t u + (u\cdot\nabla)u -\nu \Delta u+\nabla p =0,\\
		\nabla\cdot u = 0,\qquad  (t,x)\in \R_+\times\R^3,
	\end{cases}
\end{equation}
with initial data $u_0$ belonging to some critical spaces.
In fact, it is easy to see that if $(u,p)(t,x)$ is a solution of incompressible Navier–Stokes equations, then $(u_\lambda,p_\lambda)(t,x)$ with the scaling property
\begin{equation}\label{eq1.11}
	\left(u_\lambda,p_\lambda\right)(t,x)=\left(\lambda u(\lambda^2 t,\lambda x),\lambda^2p(\lambda^2 t,\lambda x)\right)
\end{equation}
also solves the system \eqref{eq-ns} with initial data $u_\lambda(0,x)=\lambda u_0(\lambda x)$.
Hence, in the literature concerning the theory of Navier–Stokes equations, the function spaces which are invariant under the scaling \eqref{eq1.11} are called {\em critical spaces}. For instance, typical examples for local (or global) well-posedness which holds for small initial data can be outlined as follows:
\begin{equation}\label{scale}
	\dot{H}^{\frac{1}{2}}(\R^3) \hookrightarrow L^3(\R^3) \hookrightarrow \dot{B}_{p,\infty}^{-1+\frac{3}{p}}(\R^3) \hookrightarrow {\rm BMO}^{-1}(\R^3) \hookrightarrow \dot{B}_{\infty,\infty}^{-1}(\R^3),
\end{equation}
where $3<p<\infty$.

The first result was obtained by Fujita and Kato \cite{fujita-kato-64} in the sense of {\em mild solution}. Based on the theory of semi-groups of operators, they turned the system \eqref{eq-ns} into the study of an integral equation, and proved the global well-posedness for small initial data in $\dot{H}^{\frac{1}{2}}$ by the Banach contraction principle. Two decades later, the study of the Navier–Stokes equations in critical space $L^3$ was continued by Kato in \cite{kato1984}; see also \cite{giga1986}. After that, the existence of global solutions from small data in $\dot{B}_{p,\infty}^{-1+\frac{3}{p}}$ was proved by Cannone and Planchon \cite{CannonePhD,Planchon98}. In $2001$, Koch and Tataru \cite{Koch-Tataru} addressed the global well-posedness of the Navier–Stokes equations with small initial data in ${\rm BMO}^{-1}$, where the constructed solutions also lie in the class $L_{loc}^2$.
In particular, they employed Kato's strategy in constructing a mild solution, the linear part in the mild formulation of \eqref{eq-ns} must enjoy the same property, i.e., $\|e^{t\Delta}u_0\|_{L_{loc}^2}<+\infty$.
This together with the parabolic scaling as defined above
%an important insight is that $u\in L_{loc}^2$, even in the sense of distributions.
% the mild formulation of \eqref{eq-ns} implies that the linear part $\|e^{t\Delta}u_0\|_{L_{loc}^2}$ must be finite, which
specifically provided that $u_0\in {\rm BMO}^{-1}$, in the sense of caloric extension.
%(i.e. convolution with the heat kernel)
%; see \eqref{def-bmo-1} below.
%Hereafter, the solution constructed similar to Koch and Tataru \cite{Koch-Tataru} is referred to as {\rm Koch-Tataru solution}.

Up to now, it seems natural to argue that ${\rm BMO}^{-1}$ is the largest critical space among those listed in \eqref{scale} where positive results are known.
Indeed, Bourgain and Pavlovi\'{c} \cite{bourgain07} proved that the Cauchy problem for the incompressible Navier–Stokes equations is ill-posed in $\dot{B}_{\infty,\infty}^{-1}$ in the sense of {\em norm inflation}. Namely, the solution of \eqref{eq-ns} evolving from any small data in $\dot{B}_{\infty,\infty}^{-1}$ can become arbitrarily large after a short time.
On the other hand, Germain \cite{Germain08} proved that the solution map of Navier–Stokes equations is not $C^2$ in $\dot{B}_{\infty,q}^{-1}$ for $q>2$. The ill-posed problem in this class was later investigated by Yoneda in \cite{Yoneda-2010}. After that, in a very interesting paper \cite{ill-wang-2015}, Wang showed the solution of \eqref{eq-ns} is also instable in $\dot{B}_{\infty,q}^{-1}$ ($1\leq q\leq 2$) by showing norm inflation phenomena. Very recently, Coiculescu and Palasek \cite{nonuniqueness-BMO}
further addressed the nonuniqueness of smooth solutions for the Navier–Stokes equations with initial data in ${\rm BMO}^{-1}$ via the convex integration technique. Their result implies the sharpness of the Koch-Tataru theorem \cite{Koch-Tataru}; in other words, smallness is necessary for global regularity.

%Since active nematic fluids can exhibit a long-range orientational ordering,
Moreover, we turn to review some existing results for the dynamical vector models of nematic liquid crystals that stem from Ericksen-Leslie theory. Due to numerous progress in this direction, we do not intend to cover all topics but concentrate on the well-posedness theory of the simplified Ericksen-Leslie system, as proposed by Lin and Liu \cite{Lin-Liu-1,Lin-Liu-2}, particularly with rough initial data. That is
\begin{equation}\label{eq1.13}
	\begin{cases}
		\partial_tu+(u\cdot\nabla)u -\Delta u + \nabla p =-\nabla\cdot \left( \nabla d\otimes \nabla d \right);\\
		\partial_t d + (u\cdot\nabla)d =\Delta d + |\nabla d|^2d\\
		\nabla\cdot u = 0,\\
		(u,d)(t,x)\mid_{t=0}=(u_0,d_0)(x), \quad |d_0(x)|=1,
	\end{cases}
\end{equation}
where $(t,x)\in \R_+\times \R^n$, $u$ is again the velocity of flow, $d(t,\cdot): \R^n\rightarrow S^2$, the unit sphere in $\R^3$, is a unit-vector field that represents the orientation of molecules in nematic liquid crystals, and $p$ is the pressure function. Similarly, $\nabla d\otimes \nabla d$ is a $3\times 3$ matrix whose $(\alpha,\beta)$-component is given by $\partial_\alpha d \cdot \partial_\beta d$. Note that the above system is also invariant under the following parabolic scaling:
\[
\left(u_\lambda,d_\lambda,p_\lambda\right)(t,x)=\left(\lambda u(\lambda^2 t,\lambda x),d(\lambda^2 t,\lambda x),\lambda^2p(\lambda^2 t,\lambda x)\right).
\]
Motivated by the seminal work \cite{Koch-Tataru}, Wang \cite{BMO-2014} addressed the existence of Koch-and-Tataru-type solution to system \eqref{eq1.13} in $\R^n$ with rough data $(u_0,d_0)\in {\rm BMO}^{-1} \times {\rm BMO}$.
It was later shown by Du and Wang in \cite{BMO-regularity,Du-BMO-2014} that this solution has arbitrary space-time regularity.
%The space-time regularity of this solution was later shown by Du and Wang in \cite{BMO-regularity,Du-BMO-2014}.
As for the $Q$-tensor based models of nematic liquid crystals, however, it is still unknown whether a similar result holds.

In this paper, we will focus on the well-posed problem for active system \eqref{eq1.1} with rough initial data.
Observe that system \eqref{eq1.1} is invariant by the scaling
\begin{equation}
	\left(Q_\lambda,u_\lambda,p_\lambda\right)(t,x)=\left( Q(\lambda^2 t,\lambda x),\lambda u(\lambda^2 t,\lambda x),\lambda^2p(\lambda^2 t,\lambda x)\right).
\end{equation}
%in the critical space, i.e.,
This motivates us to seek a solution evolved from the initial data $(Q_0, u_0)\in L^\infty\times {\rm BMO}^{-1}$.
Let us note that the consideration of functional space for $Q_0$ is related to our forthcoming work where we are interested in whether
%topological defects will blow up in finite time
finite-time blowup of \eqref{eq1.1} comes from topological defects. Indeed, the authors in \cite{defect-active} utilized the topological defects in active matter to perform logic operations. Hence, these series of works may give an insight into the evolution of topological defects after formation.

%vortex rings and vortex filaments for the initial data.

\subsection{Main result}

In order ro state the main results, we briefly recall some functional spaces used throughout this paper. Let $\alpha\in (0,1)$, a function $f$ on $\R^n$ is said to be H\"{o}lder continuous of order $\alpha$ (simply written as $\dot C^\alpha$), if
\[
\sup_{x\ne y}\frac{|f(x)-f(y)|}{|x-y|^\alpha}< +\infty.
\]
Then, we define the H\"{o}lder semi-norm as
\[
\|f\|_{\dot{C}^\alpha}= \sup_{x\ne y}\frac{|f(x)-f(y)|}{|x-y|^\alpha}.
\]
Similar to \cite{Koch-Tataru,BMO-2014}, we define a function belongs to ${\rm BMO}$ by its caloric extension, which is equivalent to the standard definition (see Stein \cite{Stein-book}). More precisely, let $w$ be the solution of linear heat equation
\[
\partial_tw-\Delta w=0,
\]
with initial data $f$. Then, $w$ can be expressed by
\[
w=\Phi(t,x)\ast f,
\]
where $\Phi(t,x)=(4\pi t)^{-\frac{n}{2}}e^{-\frac{|x|^2}{4}}$ is the fundamental solution of heat equation. In what follows, we write $w$ in terms of the heat semigroup $e^{t\Delta}$ for simplicity ($w=e^{t\Delta}f$).

\begin{definition}
	For $0<R\leq +\infty$, we say that the function $f\in L_{loc}^1(\R^n)$ is in ${\rm BMO}$ if the semi-norm $[f]_{\rm BMO}$ is finite, i.e.,
	\begin{equation}
		[f]_{\rm BMO}:=\sup_{x_0,R}\left(R^{-\frac{n}{2}}\int_0^R\int_{B(x_0,\sqrt{R})}|\nabla w|^2 dy dt\right)^{\frac{1}{2}}<+\infty.
	\end{equation}
	Similarly, we say $f$ is in ${\rm BMO}^{-1}$ if the norm $\|f\|_{{\rm BMO}^{-1}}$ is finite, i.e.,
	\begin{equation}
		\|f\|_{{\rm BMO}^{-1}}:=\sup_{x_0,R}\left(R^{-\frac{n}{2}}\int_0^R\int_{B(x_0,\sqrt{R})}|\ w|^2 dy dt\right)^{\frac{1}{2}}<+\infty.
	\end{equation}
\end{definition}
Since the divergence of a vector field with components in ${\rm BMO}$ is obviously in
${\rm BMO}^{-1}$. One can asserts that $f\in L_{loc}^1$ is in ${\rm BMO}^{-1}$ if and only if there exist $f_i\in {\rm BMO}$ ($1\leq i\leq n$) with $f=\sum_{i=1}^{n}\partial_i f_i$; see \cite[Theorem~1]{Koch-Tataru} for details.

Notice that system \eqref{eq1.1} has more complicated nonlinear structure than \eqref{eq1.13}. Thus, we introduce the following new functional space to handle the orientation field $Q$ in our case.
\begin{definition}
	Let $0<T\leq+\infty$ and $f$ be a function defined on $[0,T)\times\R^3$, we say that $f\in \mathbb{X}_T$ if and only if
		\begin{equation}
			\begin{split}
				\|f\|_{\mathbb{X}_T}=&\sup_{t\in[0,T]}\| f(t)\|_{L^\infty}+t^{\frac{1}{4}}||f(t)||_{\dot C^{1/2}}+ t^{1/2}	\|\nabla f(t)\|_{L^\infty}\\
				&+ \sup_{x_0\in\R^3}\sup_{0<R\leq T}\left(R^{-3/2}\int_0^R\int_{B(x_0,\sqrt{R})}|\nabla f(y,t)|^2 dy dt\right)^{\frac{1}{2}}<+\infty
			\end{split}
		\end{equation}
\end{definition}
In comparison with the choice of norm for vector theory presented by Wang \cite{BMO-2014}, it is essential to include $t^{\frac{1}{4}}||f(t)||_{\dot C^{1/2}}$ to derive the estimates of certain nonlinear terms. However, the functional space for the velocity field $u$ is the same as that in \cite{BMO-2014}. Therefore,
%inspired by Koch-Tataru \cite{Koch-Tataru} and Wang \cite{BMO-2014}, the functional space for velocity $u$ can be defined as follows:
\begin{definition}
	Let $0<T\leq+\infty$ and $g$ be a function defined on $[0,T)\times\R^3$, we say that $g\in \mathbb{Y}_T$ if and only if
		\begin{equation}
			\|g\|_{\mathbb{Y}_T}=\sup_{t\in[0,T]} t^{1/2}\|g(t)\|_{L^\infty} +
			\sup_{x_0\in\R^3}\sup_{0<R\leq T}\left(R^{-3/2}\int_0^R\int_{B(x_0,\sqrt{R})}|g(y,t)|^2 dy dt\right)^{\frac{1}{2}}<+\infty.
		\end{equation}
\end{definition}
%the definition of ${\rm BMO}$ space.
%\begin{definition}
%	For $0<R\leq +\infty$, a function 
%\end{definition}
%Moreover, we denote $A\lesssim B$ if there exists a universal constant $C$ such that $A\leq CB$.
It is obvious that $\mathbb{X}_T$ and $\mathbb{Y}_T$ are both Banach spaces. Then our main result in this paper is:
\begin{theorem}\label{thm1.1}
	There exists a small positive constant $\epsilon_0$ such that if the initial data $(Q_0, u_0)\in L^\infty\times {\rm BMO}^{-1}$ satisfying 
	\begin{equation}
		\|Q_0\|_{L^\infty}+ \|u_0\|_{{\rm BMO}^{-1}}\leq \epsilon_0,
	\end{equation}
	then the system \eqref{eq1.1} admits a unique local solution  $(Q,u)\in \mathbb{X}_{T}\times\mathbb{Y}_{T}$ in $[0,T]$ satisfying
	\begin{equation}
		\|Q\|_{\mathbb{X}_{T}} + \|u\|_{\mathbb{Y}_{T}} \leq C\epsilon_0,
	\end{equation}
	where $C$ is a constant depending only on $|a|$, $b$, $c$, $\mu$, $\kappa$, $\lambda$ and $\Gamma$.
\end{theorem}
Broadly speaking, our method is based on Kato's method for constructing mild solutions from the initial data. To be more specific, we convert the differential equations into integral equations via the heat kernel. On this basis, we build an iteration scheme and show the desired result by Banach fixed-point theorem.
Compared with the earlier works $\cite{Koch-Tataru,BMO-2014}$ for different fluid models, in our case, system \eqref{eq1.1} involves more complicated interaction between $Q$ and $u$ which further complicates the analysis.
Moreover, we mention our recent work \cite{ALC-decay}, where by making use of the kernel of linearized system of \eqref{eq1.1} at equilibrium state $(0,0)$, the decay properties for classical solutions of \eqref{eq1.1} have been addressed by a different integral formula of solution.

%\subsection{Plan of this paper}
The rest of this paper is organized as follows.
In Section~\ref{sec-2}, we review basic estimates of the heat kernel and reformulate the linear problem in our setting.
Thereafter, we prove Theorem~\ref{thm1.1} in Section~\ref{sec-3}.

\section{Preliminaries}\label{sec-2}
In this section, we recall some of the notations in this paper. For brevity, we denote the letter $C$ to represent generic positive constants that may vary throughout the proofs.
Moreover, we also denote $A\lesssim B$ if there exists a universal constant $C$ such that $A\leq CB$. Furthermore, we also use the following convention: we say that a pair $(Q,u)\in \mathbb{E}_T:=\mathbb{X}_T \times \mathbb{Y}_T$ if and only if $\|Q\|_{\mathbb{X}_T}+ \|u\|_{\mathbb{Y}_T}<+\infty.$

Let $\mathbb{P}$ be the well-known Leray projector, we then have the following pointwise bound for the Oseen kernel $\mathbb{K}(x,t):= \nabla^{k+1}\mathbb{P}e^{t\Delta}$ (see \cite[Lemma~2.1]{Du-BMO-2014} and references therein).
\begin{lemma}
	Let $n$ be the dimension of space, then it holds that
	\begin{equation}
		\mathbb{K}(x,t)\leq C(k)\dfrac{1}{(\sqrt{t}+ |x|)^{n+k+1}},
	\end{equation}
	where $C$ is a positive constant depending on $k$. 
\end{lemma}

%Then, we define the following function spaces for studying system \eqref{eq1.1}:
%\begin{definition}
%	We say $(Q,u)\in \mathbb{E}_T:=\mathbb{X}_T \times \mathbb{Y}_T$ iff $\|Q\|_{\mathbb{X}_T}+ \|u\|_{\mathbb{Y}_T}<+\infty.$
%%	\begin{align*}
%%		\|(Q,u)\|_{X_T}&=\sup_{t\in[0,T]} t^{1/2}	\|(u,\nabla Q)(t)\|_{L^\infty}+\| Q(t)\|_{L^\infty}+t^{\frac{1}{4}}||Q(t)||_{\dot C^{1/2}} %+t^{\frac{3}{4}}\|(u,\nabla Q)(t)\|_{\dot C^{1/2}}
%%		\\&+
%%		\sup_{x_0\in\R^3}\sup_{0<R\leq T}\left(R^{-3/2}\int_0^R\int_{B(x_0,\sqrt{R})}|(u,\nabla Q)(y,t)|^2 dy dt\right)^{\frac{1}{2}}
%%	\end{align*}
%\end{definition}

Now, we turn to study the following linear system from \eqref{eq1.1} with initial data $(Q_0,u_0)\in L^\infty\times {\rm BMO}^{-1}$, i.e.,
\begin{equation}\label{eq-linear}
	\begin{cases}
		\partial_tQ-\Gamma\Delta Q=0,\\
		\partial_tu-\mu\Delta u=0,\\
		(Q,u)\mid_{t=0}=(Q_0,u_0).
	\end{cases}
\end{equation}
Without loss of generality, we will set $\Gamma=\mu=1$. Then, the solution $(Q_L,u_L)$ of system \eqref{eq-linear} can be expressed by
\[
Q_L=e^{t\Delta}Q_0 \quad \textbf{\rm and}\quad u_L=e^{t\Delta}u_0.
\]
Then, we have the following result:
\begin{prop}\label{prop2.1}
	For any initial data $(Q_0,u_0)\in L^\infty\times {\rm BMO}^{-1}$, we have
	\begin{equation}\label{eq1.2}
		\|(Q_L,u_L)\|_{\mathbb{E}_T}\lesssim \|Q_0\|_{L^\infty} + \|u_0\|_{{\rm BMO}^{-1}}.
	\end{equation}
\end{prop}
\begin{proof}
	According to \cite[Proposition~4.1]{Germain07}, one have
	\begin{equation}
		\|e^{t\Delta}\nabla^k u_0\|_{L^\infty}\leq C(k)\|u_0\|_{{\rm BMO}^{-1}}t^{-\frac{k+1}{2}}
	\end{equation}
	and
	\begin{equation}
		\sup_{x_0,R}\dfrac{1}{|B(x_0,\sqrt{R})|}\int_0^R\int_{B(x_0,\sqrt{R})}|t^{\frac{k}{2}}\nabla^k e^{t\Delta}u_0(y)|^2 dy dt\leq C(k)\|u_0\|_{{\rm BMO}^{-1}}^2.
	\end{equation}
	Then, it follows that
	\begin{equation}
		\begin{split}
			\sup_{[0,T]} t^{1/2} \|(Q_L,u_L)(t)\|_{L^\infty} +
			\sup_{x_0,R}&\left(R^{-3/2}\int_0^R\int_{B(x_0,\sqrt{R})}|(\nabla Q_L,u_L)(y,t)|^2 dy dt\right)^{\frac{1}{2}}\\
			&\lesssim \|u_0\|_{{\rm BMO}^{-1}}+[Q_0]_{{\rm BMO}}.
		\end{split}
	\end{equation}
	Also, by Young's convolution inequality, we get
	\begin{equation}
		\sup_{[0,T]}  \|Q_L(t)\|_{L^\infty}  \leq \|Q_0\|_{L^\infty}.
	\end{equation}
	%Finally,
	%we turn to show $t^{\frac{3}{4}}\|\nabla Q\|_{\dot{C}^{\frac{1}{2}}}$ can be bounded by $\|Q_0\|_{L^\infty}$.
%	Recalling the fact that
%	\begin{equation}
%		\mathbb{K}(x,t)\leq C(k)\dfrac{1}{(\sqrt{t}+ |x|)^{n+k+1}},
%	\end{equation}
%	where $\mathbb{K}(x,t):= \nabla^{k+1}\mathbb{P}e^{t\Delta}$ and $\mathbb{P}$ is the well-known Leray projection (see \cite[Lemma~2.1]{BMO-2014}).
	We then estimate the remainder as follows:
	\begin{equation}\label{eq2.9}
		\begin{split}
			t^{\frac{1}{4}}\| Q_L\|_{\dot{C}^{1/2}}=& t^{\frac{1}{4}}\sup_{x\ne y}\dfrac{\left| Q_L(x)- Q_L(y) \right|}{|x-y|^{\frac{1}{2}}}\\
			=& t^{\frac{1}{4}}\sup_{x\ne y}\left(\dfrac{\left|  Q_L(x)-  Q_L(y) \right|}{|x-y|}\right)^{\frac{1}{2}}\cdot \left|  Q_L(x)- Q_L(y) \right|^{\frac{1}{2}}\\
			\lesssim &t^{\frac{1}{4}}\|\nabla e^{t\Delta}  Q_0\|_{L^\infty}^{\frac{1}{2}} \|  Q_L\|_{L^\infty}^{\frac{1}{2}} \\
			\lesssim & t^{\frac{1}{4}}[ Q_0]_{{\rm BMO}}^{\frac{1}{2}}t^{-\frac{1}{4}}\| Q_0\|_{L^\infty}^{\frac{1}{2}}\\
			\lesssim &\| Q_0\|_{L^\infty}.
		\end{split}
	\end{equation}
%	\begin{equation}
%		\begin{split}
%			t^{\frac{3}{4}}\|\nabla Q_L\|_{\dot{C}^{1/2}}=& t^{\frac{3}{4}}\sup_{x\ne y}\dfrac{\left| \nabla Q_L(x)- \nabla Q_L(y) \right|}{|x-y|^{\frac{1}{2}}}\\
%			=& t^{\frac{3}{4}}\sup_{x\ne y}\left(\dfrac{\left| \nabla Q_L(x)- \nabla Q_L(y) \right|}{|x-y|}\right)^{\frac{1}{2}}\cdot \left| \nabla Q_L(x)- \nabla Q_L(y) \right|^{\frac{1}{2}}\\
%			\lesssim &t^{\frac{3}{4}}\|\nabla e^{t\Delta} \nabla Q_0\|_{L^\infty}^{\frac{1}{2}} \| e^{t\Delta} \nabla Q_0\|_{L^\infty}^{\frac{1}{2}} \\
%			\lesssim & t^{\frac{3}{4}}[ Q_0]_{{\rm BMO}}^{\frac{1}{2}}t^{-\frac{1}{2}}[ Q_0]_{{\rm BMO}}^{\frac{1}{2}}t^{-\frac{1}{4}}\\
%			\lesssim &[ Q_0]_{{\rm BMO}}.
%		\end{split}
%	\end{equation}
	Collecting all the above estimates, we get \eqref{eq1.2} and this completes the proof. 
	%$$
	%\|(Q_L,u_L)\|_{X_T}\leq C\left(\|u_0\|_{{\rm BMO}^{-1}}+\|Q_0\|_{L^\infty}\right).
	%$$
\end{proof}
%\begin{align*}
%	=& 4t^{-\frac{1}{2}}\left\|\int_{\frac{t}{8}}^{\frac{3t}{8}} e^{(t-\tau)\Delta}e^{\tau\Delta}u_0 d\tau \right\|_{L^\infty}\\
%	\lesssim & t^{-\frac{1}{2}}\left\|\int_{\frac{t}{8}}^{\frac{3t}{8}} \int_{\R^3} \dfrac{1}{(t-\tau)^{\frac{3}{2}}}e^{-\frac{(x-y)^2}{4(t-\tau)}}e^{\tau\Delta}u_0 dyd\tau \right\|_{L^\infty}\\
%	\lesssim & t^{-\frac{1}{2}}\left\|\int_{\frac{t}{8}}^{\frac{3t}{8}} \left(\int_{\R^3} \dfrac{1}{(t-\tau)^{\frac{3}{2}}}e^{-\frac{(x-y)^2}{4(t-\tau)}}\left(e^{\tau\Delta}u_0\right)^2 dy\right)^{\frac{1}{2}}d\tau \right\|_{L^\infty}\\
%	\lesssim & t^{-\frac{1}{2}}\left\|\left(\int_{\frac{t}{8}}^{\frac{3t}{8}} \int_{\R^3} \dfrac{1}{(t-\tau)^{\frac{3}{2}}}e^{-\frac{(x-y)^2}{4(t-\tau)}}\left(e^{\tau\Delta}u_0\right)^2 dyd\tau\right)^{\frac{1}{2}} \right\|_{L^\infty}\cdot\sqrt{t}\\
%	\lesssim & \left\|\left(\sum_{m=0}^\infty e^{-\frac{m}{4}}t^{-\frac{3}{2}}\int_{\frac{t}{8}}^{\frac{3t}{8}} \int_{m\leq\frac{|x-y|}{\sqrt{t}}\leq m+1}  \left(e^{\tau\Delta}u_0\right)^2 dyd\tau\right)^{\frac{1}{2}}\right\|_{L^\infty}\\
%	\lesssim &\sup_{x,t1>0}\left(\sum_{m=0}^\infty e^{-\frac{m}{4}}t^{-\frac{3}{2}}\int_0^t \int_{B(x,\sqrt{t})}  \left(e^{\tau\Delta}u_0\right)^2 dyd\tau\right)^{\frac{1}{2}}\\
%	\lesssim &\|u_0\|_{{\rm BMO}^{-1}}
%\end{align*}

\section{Proof of Theorem~\eqref{thm1.1}}\label{sec-3}
In this section, we shall prove our main result, which implies the local-in-time existence of a Koch-Tataru solution for the incompressible active nematic liquid crystals in $\R^3$. To do this, we write the system \eqref{eq1.1} as an integral system in the following:
\begin{numcases}{}
	Q= e^{t\Delta}Q_0 -\int_0^t e^{(t-\tau)\Delta}F_1(Q,u)d\tau,\\
	u= e^{t\Delta}u_0 -\int_0^t e^{(t-\tau)\Delta}\mathbb{P}\nabla\cdot F_2(Q,u)d\tau,
\end{numcases} 
where
\begin{equation}
	\begin{split}
		F_1(Q,u)&:= (u\cdot\nabla)Q + Q\Omega(u)-\Omega(u) Q-\lambda |Q|D(u)-\Gamma \widetilde{H}[Q], \\
		F_2(Q,u)&:= u\otimes u+ \nabla Q \odot\nabla Q+ \Delta QQ-Q\Delta Q +\lambda (|Q|\Delta Q+|Q|\widetilde{H}[Q])-\kappa  Q,\\
		\widetilde{H}[Q]&:= -aQ+b\left[Q^2-\dfrac{{\rm Tr}(Q^2)}{3}\mathbb{I}_3\right]-cQ{\rm Tr}(Q^2).
	\end{split}
\end{equation}
Since the basic tool is the Banach contraction principle, we define a solution map
\[
\bm{\mathcal{S}}: (\tilde{Q},\tilde{u})\in \mathbb{E}_T \mapsto (Q,u)
\]
with
\[
(Q,u)=\bm{\mathcal{S}}(\tilde{Q},\tilde{u})= \left(S_1(\tilde{Q},\tilde{u}),S_2(\tilde{Q},\tilde{u})\right)
\]
and
\begin{numcases}{}
	Q=S_1(\tilde{Q},\tilde{u})= e^{t\Delta}Q_0 -\int_0^t e^{(t-\tau)\Delta}F_1(\tilde{Q},\tilde{u})d\tau,\label{eq3.4}\\
	u=S_2(\tilde{Q},\tilde{u})= e^{t\Delta}u_0 -\int_0^t e^{(t-\tau)\Delta}\mathbb{P}\nabla\cdot F_2(\tilde{Q},\tilde{u})d\tau,\label{eq3.5}
\end{numcases}
Moreover, for $\epsilon>0$, we define the ball $\mathcal{B}_\epsilon(Q_L,u_L)$ in $\mathbb{E}_T$ with center $(Q_L,u_L)$ and radius $\epsilon$ by
\[
\mathcal{B}_\epsilon(Q_L,u_L)=\left\lbrace (Q,u)\in \mathbb{E}_T \mid \|Q-Q_L\|_{\mathbb{X}_T}+ \|u-u_L\|_{\mathbb{Y}_T}\leq \epsilon \right\rbrace.
\]
Then we have the following two lemmas.
\begin{lemma}\label{lem3.1}
	There exists $\epsilon_0>0$ such that if
	\[
	\|Q_0\|_{L^\infty}+ \|u_0\|_{{\rm BMO}^{-1}}\leq \epsilon_0,
	\]
	Then $\bm{\mathcal{S}}$ maps $\mathcal{B}_{\epsilon_0}(Q_L,u_L)$ to itself.
\end{lemma}
%We will prove that $S$ is a contraction mapping for $E_0=\|u_0\|_{{\rm BMO}^{-1}}+\|Q_0\|_{L^\infty}$ small enough.

\begin{proof}
	In view of the formulation \eqref{eq3.4} and \eqref{eq3.5}, we only need to show that
	\[
	\underbrace{\left\| \int_0^t e^{(t-\tau)\Delta}F_1(\tilde{Q},\tilde{u})d\tau \right\|_{\mathbb{X}_T}}_{\mathcal{I}_Q} + \underbrace{\left\|\int_0^t e^{(t-\tau)\Delta}\mathbb{P}\nabla\cdot F_2(\tilde{Q},\tilde{u})d\tau  \right\|_{\mathbb{Y}_T}}_{\mathcal{I}_u} \leq \epsilon_0.
	\]
	\textbf{Step~1}. In this step, we shall estimate $\mathcal{I}_u$. By the definition of norm $\mathbb{Y}_T$, it is sufficient to verify that it holds
	\begin{equation}\label{eq3.6}
		\begin{split}
			t^{\frac{1}{2}}\left\| \int_0^t e^{(t-\tau)\Delta}\mathbb{P}\nabla \cdot F_2(\tilde{Q},\tilde{u})d\tau \right\|_{L^\infty}\hspace{9.5cm}&\\
			+ \sup_{x_0\in\R^3}\sup_{0<R\leq T}\left(R^{-3/2}\int_0^R\int_{B(x_0,\sqrt{R})}\left(\int_0^t e^{(t-\tau)\Delta}\mathbb{P}\nabla\cdot F_2(\tilde{Q},\tilde{u})d\tau \right)^2 dy dt\right)^{\frac{1}{2}}\lesssim \epsilon_0.
		\end{split}
	\end{equation}
	For the first term on the left hand side of \eqref{eq3.6}, we estimate it in two parts:
	\begin{align}
	t^{\frac{1}{2}}\left\| \int_0^t e^{(t-\tau)\Delta}\mathbb{P}\nabla \cdot F_2(\tilde{Q},\tilde{u})d\tau \right\|_{L^\infty}=t^{\frac{1}{2}}\left\| \int_0^{\frac{t}{2}} +\int_{\frac{t}{2}}^{t}\right\|_{L^\infty} \overset{\Delta}{=} I_{11}+I_{12}.
	\end{align}
%	Recalling the fact that
%	\begin{equation}
%	\mathbb{K}(x,t)\leq C(k)\dfrac{1}{(\sqrt{t}+ |x|)^{n+k+1}},
%	\end{equation}
%	where $\mathbb{K}(x,t):= \nabla^{k+1}\mathbb{P}e^{t\Delta}$ and $\mathbb{P}$ is the Leray projection (see \cite[Lemma~2.1]{BMO-2014}).
	%Then,
	According to Proposition~\ref{prop2.1}, together with the following identities 
	\begin{align*}
		\left(\Delta \tilde{Q}\tilde{Q}-\tilde{Q}\Delta \tilde{Q}\right)_{\alpha\beta}=&\sum_j {\rm div} \left( \tilde{Q}_{\alpha j}\nabla \tilde{Q}_{j\beta} - \nabla \tilde{Q}_{\alpha j} \tilde{Q}_{j\beta}\right),\\
		|\tilde{Q}|\Delta \tilde{Q}=&{\rm div}\left(|\tilde{Q}|\nabla \tilde{Q}\right)-\nabla |\tilde{Q}|\cdot\nabla \tilde{Q},
	\end{align*}
	we have
	\begin{align*}
	I_{11}\lesssim& t^{\frac{1}{2}}\left\|  \int_0^{\frac{t}{2}} \nabla\mathbb{P}e^{(t-\tau)\Delta} \left( \tilde{u}\otimes \tilde{u}+ \nabla \tilde{Q} \odot\nabla \tilde{Q} - \nabla |\tilde{Q}| \cdot\nabla \tilde{Q} +\sum_{l=1}^{4}|\tilde{Q}|^l \right)d\tau \right\|_{L^\infty}\\
	& +t^{\frac{1}{2}}\left\|  \int_0^{\frac{t}{2}} \nabla^2\mathbb{P}e^{(t-\tau)\Delta} \left(\nabla \tilde{Q}\tilde{Q}-\tilde{Q}\nabla \tilde{Q} + |\tilde{Q}|\nabla \tilde{Q} \right)d\tau \right\|_{L^\infty}\\
	\lesssim& t^{\frac{1}{2}}\left\|  \int_0^{\frac{t}{2}}\int_{\R^3}\dfrac{1}{\left(\sqrt{t-\tau} +|x-y| \right)^{4}}\left( |\tilde{u}|^2 + |\nabla \tilde{Q}|^2+ \sum_{l=1}^{4}|\tilde{Q}|^l \right)dyd\tau \right\|_{L^\infty}\\
	&+t^{\frac{1}{2}}\left\|  \int_0^{\frac{t}{2}}\left|\int_{\R^3}\dfrac{1}{\left(\sqrt{t-\tau} +|x-y| \right)^{5}}dy\right|\|\tilde{Q}\|_{L^\infty}\|\nabla \tilde{Q}\|_{L^\infty}d\tau \right\|_{L^\infty}\\
	\lesssim& \left\| t^{-\frac{3}{2}} \int_0^{\frac{t}{2}}\int_{\R^3}\dfrac{1}{\left(1 +\frac{|x-y|}{\sqrt{t-\tau}} \right)^{4}}\left( |\tilde{u}|^2 + |\nabla \tilde{Q}|^2 + \sum_{l=1}^{4}|\tilde{Q}|^l \right)dyd\tau \right\|_{L^\infty} +t^{\frac{1}{2}}\left\|  \int_0^{\frac{t}{2}}\dfrac{1}{t-\tau}\dfrac{\|\tilde{Q}\|_{\mathbb{X}_T}^2}{\sqrt{\tau}}d\tau \right\|_{L^\infty}\\
	\lesssim& \left\| t^{-\frac{3}{2}}\sum_{m=0}^{+\infty} \int_0^{\frac{t}{2}}\int_{m\leq \frac{|x-y|}{\sqrt{t-\tau}}\leq m+1}\dfrac{|\tilde{u}|^2 + |\nabla \tilde{Q}|^2+ \sum_{l=1}^{4}|\tilde{Q}|^l}{(1 +m )^{4}} dyd\tau \right\|_{L^\infty}+ \|\tilde{Q}\|_{\mathbb{X}_T}^2\left\|  \int_0^{\frac{t}{2}}\dfrac{1}{t-\tau}\dfrac{\sqrt{t}}{\sqrt{\tau}}d\tau \right\|_{L^\infty}\\
	\lesssim& \left\| \sum_{m=0}^{+\infty}\dfrac{m^2}{(1 +m )^{4}} t^{-\frac{3}{2}}\int_0^{\frac{t}{2}}\int_{B(x,\sqrt{t})}\left(|\tilde{u}|^2 + |\nabla \tilde{Q}|^2 + \sum_{l=1}^{4}|\tilde{Q}|^l\right) dyd\tau \right\|_{L^\infty}+ \|\tilde{Q}\|_{\mathbb{X}_T}^2   \int_0^{\frac{t}{2}}\dfrac{t^{-1}}{\left(1-\frac{\tau}{t}\right)\cdot\sqrt{\frac{\tau}{t}}} d\tau \\
	\lesssim& \left\| \sum_{m=0}^{+\infty}\dfrac{m^2}{(1 +m )^{4}} t^{-\frac{3}{2}}\int_0^{\frac{t}{2}}\int_{B(x,\sqrt{t})}  \sum_{l=1}^{4}|\tilde{Q}|^l dyd\tau \right\|_{L^\infty}+\|\left(\tilde{Q},\tilde{u}\right)\|_{\mathbb{E}_T}^2  + \|\tilde{Q}\|_{\mathbb{X}_T}^2  \int_0^{\frac{1}{2}}\dfrac{1}{\left(1-s\right)\cdot\sqrt{s}}ds\\
	\lesssim&  \|\left(\tilde{Q},\tilde{u}\right)\|_{\mathbb{E}_T}^2  + t\sum_{l=1}^4\| \tilde{Q}\|_{L^\infty}^l
    \end{align*}
    where we also used the fact that $|\nabla |\tilde{Q}||\leq |\nabla \tilde{Q}|$. Similarly, we can get
    \begin{align*}
	I_{12} 
	\lesssim& t^{\frac{1}{2}}\left\|  \int_{\frac{t}{2}}^t\left|\int_{\R^3}\dfrac{1}{\left(\sqrt{t-\tau} +|x-y| \right)^{4}}dy\right|\left(\|\tilde{u}\|_{L^\infty}^2+ \|\nabla \tilde{Q}\|_{L^\infty}^2 +\sum_{l=1}^4 \|\tilde{Q}\|_{L^\infty}^l   \right)d\tau \right\|_{L^\infty}\\
	&+ t^{\frac{1}{2}}\left\|  \int_{\frac{t}{2}}^t \left|\int_{\R^3}\dfrac{1}{\left(\sqrt{t-\tau} +|x-y| \right)^{5}}dy\right| |\tilde{Q}| |\nabla \tilde{Q}| d\tau \right\|_{L^\infty}\\
	\lesssim&  t^{\frac{1}{2}}\left\|  \int_{\frac{t}{2}}^t\dfrac{1}{\sqrt{t-\tau}}\dfrac{\|\left(\tilde{Q},\tilde{u}\right)\|_{\mathbb{E}_T}^2}{\tau}d\tau \right\|_{L^\infty}+ t\sum_{l=1}^4\|\tilde{Q}\|_{L^\infty}^l + t^{\frac{1}{2}}\left\|  \int_{\frac{t}{2}}^t\dfrac{1}{\sqrt{t-\tau}^{3/2}}\| \tilde{Q}\|_{\dot{C}^{1/2}}\|\nabla\tilde{Q}\|_{L^\infty}d\tau \right\|_{L^\infty}\\
	\lesssim&  \|\left(\tilde{Q},\tilde{u}\right)\|_{\mathbb{E}_T}^2\cdot  \int_{\frac{t}{2}}^t\dfrac{t^{-1}}{\sqrt{1-\frac{\tau}{t}}\cdot\frac{\tau}{t}} d\tau + t\sum_{l=1}^4\|\tilde{Q}\|_{L^\infty}^l + t^{\frac{1}{2}}\left\|  \int_{\frac{t}{2}}^t\dfrac{1}{\sqrt{t-\tau}^{3/2}}\dfrac{\|\tilde{Q}\|_{\mathbb{X}_T}^2}{\tau^{3/4}}d\tau \right\|_{L^\infty}\\
	\lesssim&  \|\left(\tilde{Q},\tilde{u}\right)\|_{\mathbb{E}_T}^2 \int_{\frac{1}{2}}^1\dfrac{1}{ \sqrt{1-s} \cdot s}+ \dfrac{1}{ \sqrt{1-s}^{3/2} \cdot s^{3/4}}ds+ t\sum_{l=1}^4\|\tilde{Q}\|_{L^\infty}^l \\
	\lesssim&  \|\left(\tilde{Q},\tilde{u}\right)\|_{\mathbb{E}_T}^2 + t\sum_{l=1}^4
	\|\tilde{Q}\|_{L^\infty}^l.
	\end{align*}
	And thus
	\begin{equation}\label{eq3.9}
		t^{\frac{1}{2}}\left\| \int_0^t e^{(t-\tau)\Delta}\mathbb{P}\nabla \cdot F_2(\tilde{Q},\tilde{u})d\tau \right\|_{L^\infty}\lesssim \|\left(\tilde{Q},\tilde{u}\right)\|_{\mathbb{E}_T}^2 + t\sum_{l=1}^4
		\|\tilde{Q}\|_{L^\infty}^l.
	\end{equation}
	Next, in order to get the estimate of the second term on the left hand side of \eqref{eq3.6}, we introduce the following smooth cut-off function for any given $x_0\in\R^3$ and $R\in (0,T]$:
	\begin{equation}\label{eq1.6}
		\chi\left(\frac{y}{\sqrt{R}}\right):=
		\begin{cases}
			1,&\text{$|y-x_0|\leq 3\sqrt{R}$},\\
			0,&\text{$|y-x_0|\geq 5\sqrt{R}$}.
		\end{cases}
	\end{equation}
	Hence, we only need to show that
	\[
	\begin{split}
		 &\sup_{x_0\in\R^3}\sup_{0<R\leq T}\left(R^{-3/2}\int_0^R\int_{B(x_0,\sqrt{R})}\left(\nabla \int_0^t e^{(t-\tau)\Delta}\mathbb{P}\chi F_2(\tilde{Q},\tilde{u})d\tau \right)^2 dy dt\right)^{\frac{1}{2}}\\
		 +& \sup_{x_0\in\R^3}\sup_{0<R\leq T}\left(R^{-3/2}\int_0^R\int_{B(x_0,\sqrt{R})}\left(\nabla\int_0^t e^{(t-\tau)\Delta}\mathbb{P}(1-\chi)  F_2(\tilde{Q},\tilde{u})d\tau \right)^2 dy dt\right)^{\frac{1}{2}}  := I_{21} +  I_{22}\lesssim \epsilon_0.
	\end{split}
	\]
	To estimate $I_{21}$, we set up a heat equation
	\begin{equation}
		W_t-\Delta W=\chi F_2(\tilde{Q},\tilde{u}),
	\end{equation}
	with initial data $W\mid_{t=0}=0$.
	Hereafter, we will drop the Leray projector $\mathbb{P}$ for simplicity. Note that it is bounded in $L^2$ and commutes with $e^{(t-\tau)\Delta}$ and $\nabla$. Following the pioneering work by Koch and Tataru \cite{Koch-Tataru}, an energy argument implies that
	\begin{equation}\label{eq1.18}
		\int_0^R\int_{B(x_0,\sqrt{R})}	|\nabla W|^2 dyd\tau\leq \int_0^R \int_{\R^3}	|\nabla W|^2 dyd\tau \leq \int_0^R \int_{\R^3}	|\chi F_2(\tilde{Q},\tilde{u})  W| dyd\tau.
	\end{equation}
	Then, it follows from \eqref{eq1.6} that
	\[
	\int_0^R \int_{\R^3}|\chi F_2(\tilde{Q},\tilde{u})  W| dyd\tau\lesssim \|W\|_{L^\infty(Q(x_0,5\sqrt{R}))}\int_0^R \int_{\R^3}|\chi F_2(\tilde{Q},\tilde{u})| dyd\tau.
	\]
	where $Q(x_0,\sqrt{R})=B(x_0,\sqrt{R})\times (0,R)$ is the space-time ball.
	Indeed, we have
	\[
	\|W\|_{L^\infty(Q(x_0,5\sqrt{R}))}= \left\| \int_0^{\frac{t}{2}} e^{(t-\tau)\Delta}\chi F_2(\tilde{Q},\tilde{u})  d\tau + \int_{\frac{t}{2}}^t e^{(t-\tau)\Delta}\chi F_2(\tilde{Q},\tilde{u}) d\tau \right\|_{L^\infty(Q(x_0,5\sqrt{R}))}.
	\]
	Clearly, via a similar argument as $I_{11}$ and $I_{12}$, we are able to evaluate $\|W\|_{L^\infty(Q(x_0,5\sqrt{R}))}$ as follows:
	%Similarly to $I_{21}$ and $I_{22}$, we then get
	\begin{equation}\label{eq1.20}
		\begin{split}
			&\left\| \int_0^{\frac{t}{2}} e^{(t-\tau)\Delta}\chi F_2(\tilde{Q},\tilde{u})  d\tau \right\|_{L^\infty(Q(x_0,5\sqrt{R}))}\\
			\lesssim& \left\| \int_0^{\frac{t}{2}} e^{(t-\tau)\Delta}  \left(|\tilde{u}|^2 +|\nabla\tilde{Q}|^2+\sum_{l=1}^4|\tilde{Q}|^l \right)  d\tau \right\|_{L^\infty(Q(x_0,5\sqrt{R}))} + \left\| \int_0^{\frac{t}{2}} \nabla e^{(t-\tau)\Delta}  \left(|\nabla\tilde{Q}||\tilde{Q}| \right)  d\tau \right\|_{L^\infty(Q(x_0,5\sqrt{R}))}\\
			\lesssim& \left\| \dfrac{1}{\sqrt{t}^3} \sum_{m=0}^{+\infty}\int_0^{\frac{t}{2}}\int_{m\leq \frac{|x-y|}{\sqrt{t-\tau}}\leq m+1}e^{-m^2}\left( |\tilde{u}|^2  + |\nabla \tilde{Q}|^2 +\sum_{l=1}^{4}|\tilde{Q}|^l \right)dyd\tau \right\|_{L^\infty(Q(x_0,5\sqrt{R}))}\\
			& +\left\|  \int_0^{\frac{t}{2}}\dfrac{1}{\sqrt{t-\tau}}\cdot\frac{\|\tilde{Q}\|_{\mathbb{X}_T}^2}{\sqrt{\tau}}d\tau \right\|_{L^\infty(Q(x_0,5\sqrt{R}))}\\
			\lesssim&  \|\left(\tilde{Q},\tilde{u}\right)\|_{\mathbb{E}_T}^2+ t\sum_{l=1}^4\|\tilde{Q}\|_{L^\infty}^l +  \|\tilde{Q}\|_{\mathbb{X}_T}^2 \int_0^{\frac{1}{2}}\dfrac{1}{ \sqrt{1-s}  \sqrt{s}}ds \\
			\lesssim & \|\left(\tilde{Q},\tilde{u}\right)\|_{\mathbb{E}_T}^2 + t\sum_{l=1}^4
			\|\tilde{Q}\|_{L^\infty}^l,\\
%		\end{split}
%	\end{equation}
%	and
%	%Moreover, by the definition of $\chi$, we estimate
%	\begin{equation}\label{eq1.21}
%		\begin{split}
			&\left\| \int_{\frac{t}{2}}^t e^{(t-\tau)\Delta}\chi F_2(\tilde{Q},\tilde{u})  d\tau \right\|_{L^\infty(Q(x_0,5\sqrt{R}))}\\
			\lesssim & t\sum_{l=1}^4\|\tilde{Q}\|_{L^\infty}^l  +\left\|  \int_{\frac{t}{2}}^t\int_{\R^3}\dfrac{e^{-\frac{(x-y)^2}{4(t-\tau)}}}{\sqrt{t-\tau}^3}\cdot\frac{\tau}{\tau}\left( \|\tilde{u}\|_{L^\infty}^2  + \|\nabla \tilde{Q}\|_{L^\infty}^2  \right)dyd\tau \right\|_{L^\infty} + \left\|  \int_{\frac{t}{2}}^t\dfrac{1}{\sqrt{t-\tau}}\cdot\frac{\|\tilde{Q}\|_{\mathbb{X}_T}^2}{\sqrt{\tau}}d\tau \right\|_{L^\infty}\\
			\lesssim &t\sum_{l=1}^4\|\tilde{Q}\|_{L^\infty}^l  +\|\left(\tilde{Q},\tilde{u}\right)\|_{\mathbb{E}_T}^2  \left(\left\|  \int_{\frac{t}{2}}^t\frac{1}{\tau}d\tau\int_{\R^3}\dfrac{e^{-\frac{(x-y)^2}{4(t-\tau)}}}{\sqrt{t-\tau}^3}dy \right\|_{L^\infty}+  \int_{\frac{1}{2}}^1\dfrac{1}{\sqrt{1-s}\sqrt{s}} ds\right)  \\
			\lesssim & \|\left(\tilde{Q},\tilde{u}\right)\|_{\mathbb{E}_T}^2 + t\sum_{l=1}^4
			\|\tilde{Q}\|_{L^\infty}^l.
		\end{split}
	\end{equation}
	%Then, it follow from \eqref{eq1.20} and \eqref{eq1.21} that
	Thus, we infer that
	\begin{equation}\label{eq1.22}
		\|W\|_{L^\infty(Q(x_0,5\sqrt{R}))}\lesssim \|\left(\tilde{Q},\tilde{u}\right)\|_{X_T}^2 + t\sum_{l=1}^4\|\tilde{Q}\|_{L^\infty}^l.
	\end{equation}
	Next, we turn to estimate $R^{-3/2}\int_0^R \int_{\R^3}|\chi F_2(\tilde{Q},\tilde{u})| dyd\tau$.
	Similarly, one can see that
	\begin{align}\label{eq1.23}
		R^{-3/2}\int_0^R \int_{\R^3}&|\chi F_2(\tilde{Q},\tilde{u})| dyd\tau\nonumber\\
		\lesssim &R^{-3/2}\int_0^R \int_{B(x_0,5\sqrt{R})}\left(|\tilde{u}|^2 +|\nabla\tilde{Q}|^2+\sum_{l=1}^4|\tilde{Q}|^l \right) dyd\tau\nonumber\\
		& + R^{-3/2} \int_0^R  \int_{B(x_0,5\sqrt{R})}\frac{1}{\sqrt{R}}\nabla_y \chi\left(\frac{y}{\sqrt{R}}\right) |\nabla\tilde{Q}||\tilde{Q}| dyd\tau\\
		\lesssim & \|\left(\tilde{Q},\tilde{u}\right)\|_{\mathbb{E}_T}^2 + R\sum_{l=1}^4\|\tilde{Q}\|_{L^\infty}^l+  R^{-2}\int_0^R\int_{B(x_0,5\sqrt{R})}\frac{1}{\sqrt{\tau}}dyd\tau\cdot \|\tilde{Q}\|_{\mathbb{X}_T}^2\nonumber\\
		\lesssim & \|\left(\tilde{Q},\tilde{u}\right)\|_{\mathbb{E}_T}^2 + R\sum_{l=1}^4\|\tilde{Q}\|_{L^\infty}^l.\nonumber
	\end{align}
%	\begin{equation}\label{eq1.23}
%		\begin{split}
%			R^{-3/2}\int_0^R \int_{\R^3}&|\chi F_2(\tilde{Q},\tilde{u})| dyd\tau\\
%			\lesssim &R^{-3/2}\int_0^R \int_{B(x_0,5\sqrt{R})}\left(|\tilde{u}|^2 +|\nabla\tilde{Q}|^2+\sum_{l=1}^4|\tilde{Q}|^l \right) dyd\tau\\
%			& + R^{-3/2} \int_0^R  \int_{B(x_0,5\sqrt{R})}\frac{1}{\sqrt{R}}\nabla_y \chi\left(\frac{y}{\sqrt{R}}\right) |\nabla\tilde{Q}||\tilde{Q}| dyd\tau\\
%			\lesssim & \|\left(\tilde{Q},\tilde{u}\right)\|_{\mathbb{E}_T}^2 + R\sum_{l=1}^4\|\tilde{Q}\|_{L^\infty}^l+  R^{-2}\int_0^R\int_{B(x_0,5\sqrt{R})}\frac{1}{\sqrt{\tau}}dyd\tau\cdot \|\tilde{Q}\|_{\mathbb{X}_T}^2\\
%			\lesssim & \|\left(\tilde{Q},\tilde{u}\right)\|_{\mathbb{E}_T}^2 + R\sum_{l=1}^4\|\tilde{Q}\|_{L^\infty}^l.
%		\end{split}
%	\end{equation}
	Recalling that $R\leq T$, together with \eqref{eq1.18}, \eqref{eq1.22} and \eqref{eq1.23}, we then get
	\begin{equation}\label{eq3.16}
		\begin{split}
			I_{21}\lesssim & \sup_{x_0\in\R^3}\sup_{0<R\leq T}\left(R^{-3/2}\int_0^R\int_{B(x_0,\sqrt{R})}\left(\nabla\int_0^t e^{(t-\tau)\Delta}\mathbb{P}\chi  F_2(\tilde{Q},\tilde{u})d\tau \right)^2 dy dt\right)^{\frac{1}{2}}\\
			= & \sup_{x_0\in\R^3}\sup_{0<R\leq T}\left(R^{-3/2}\int_0^R\int_{B(x_0,\sqrt{R})}	|\nabla W|^2 dydt \right)^{\frac{1}{2}}\\
			\lesssim& \|\left(\tilde{Q},\tilde{u}\right)\|_{\mathbb{E}_T}^2 + T\sum_{l=1}^4\|\tilde{Q}\|_{L^\infty}^l.
		\end{split}
	\end{equation}
	For the remainder term $I_{22}$, it is easy to see that
	\begin{equation}\label{eq1.25}
		I_{22}^2\lesssim \sup_{x_0\in\R^3,0<R\leq T} R \left\|\nabla\int_0^t e^{(t-\tau)\Delta}\mathbb{P}(1-\chi) F_2(\tilde{Q},\tilde{u})d\tau\right\|_{L^\infty({Q(x_0,\sqrt{R})})}^2.
	\end{equation}
	Similarly to the proof of \cite[Proposition~3.1]{BMO-2014}, one can estimate the above term in the following two different ways.
	\begin{itemize}
		\item [(i)]  As the first case, we assume that $0\leq \tau\leq t\leq (\sqrt{R}/2)^2$ and thus
		\begin{align*}
			&\sup_{x_0\in\R^3,0< R\leq T} R \left\|\nabla\int_0^t e^{(t-\tau)\Delta}\mathbb{P}(1-\chi) F_2(\tilde{Q},\tilde{u})d\tau\right\|_{L^\infty({Q(x_0,\sqrt{R})})}^2\\
			\lesssim& \sup_{x_0\in\R^3,0< R\leq T} R \left\| \int_0^t\int_{\R^3} \dfrac{1-\chi(y/\sqrt{R})}{(\sqrt{t-\tau}+|x-y|)^4} \left(|\tilde{u}|^2 +|\nabla\tilde{Q}|^2+\sum_{l=1}^4|\tilde{Q}|^l \right)dy d\tau\right\|_{L^\infty({Q(x_0,\sqrt{R})})}^2\\
			& + \sup_{x_0\in\R^3,0< R\leq T} R \left\| \int_0^t\int_{\R^3} \dfrac{1}{(\sqrt{t-\tau}+|x-y|)^5} |\nabla\tilde{Q}||\tilde{Q}| dy d\tau\right\|_{L^\infty({Q(x_0,\sqrt{R})})}^2\\
			\lesssim& \sup_{x_0\in\R^3,0< R\leq T} R\left\|  \int_0^{t}\sum_{m=1}^{+\infty}\int_{m\sqrt{R}\leq |x-y|\leq (m+1)\sqrt{R}}\dfrac{1}{(m\sqrt{R})^4}\left( |\tilde{u}|^2  + |\nabla \tilde{Q}|^2 +\sum_{l=1}^{4}|\tilde{Q}|^l \right)dyd\tau \right\|_{L^\infty(Q(x_0,\sqrt{R}))}^2\\
			& + \sup_{x_0\in\R^3,0< R\leq T} \left\|  \int_0^{t} \dfrac{1}{\sqrt{t-\tau}}\cdot\dfrac{\sqrt{\tau}}{\sqrt{\tau}}\|\nabla \tilde{Q}\|_{L^\infty}  \|\tilde{Q}\|_{L^\infty}dyd\tau \right\|_{L^\infty(Q(x_0,\sqrt{R}))}^2\\
			\lesssim& \sup_{x_0\in\R^3,0< R\leq T} \left( \sup_{x\in B(x_0,\sqrt{R})} R^{-3/2} \int_0^{R} \int_{B(x,\sqrt{R})}( |\tilde{u}|^2  + |\nabla \tilde{Q}|^2 )dyd\tau + t\sum_{l=1}^4\|\tilde{Q}\|_{L^\infty}^l \right)^2+  \|\tilde{Q}\|_{\mathbb{X}_T}^4\\
			\lesssim & \left( \|\left(\tilde{Q},\tilde{u}\right)\|_{\mathbb{E}_T}^2 + t\sum_{l=1}^4\|\tilde{Q}\|_{L^\infty}^l \right)^2.
		\end{align*}
		\item [(ii)] As the second case, we set $(\sqrt{R}/2)^2\leq t\leq R$ and observe that
		\begin{align*}
			&\sup_{x_0\in\R^3,0< R\leq T} R \left\|\nabla\int_0^\frac{t}{2} e^{(t-\tau)\Delta}\mathbb{P}(1-\chi) F_2(\tilde{Q},\tilde{u})d\tau\right\|_{L^\infty({Q(x_0,\sqrt{R})})}^2\\
			\lesssim& \sup_{x_0\in\R^3,0< R\leq T} R \left\| \int_0^\frac{t}{2}\int_{\R^3} \dfrac{1}{(\sqrt{t-\tau}+|x-y|)^4} \left(|\tilde{u}|^2 +|\nabla\tilde{Q}|^2+\sum_{l=1}^4|\tilde{Q}|^l \right)dy d\tau\right\|_{L^\infty({Q(x_0,\sqrt{R})})}^2\\
			& + \sup_{x_0\in\R^3,0< R\leq T} R \left\| \int_0^\frac{t}{2}\int_{\R^3} \dfrac{1}{(\sqrt{t-\tau}+|x-y|)^5} |\nabla\tilde{Q}||\tilde{Q}| dy d\tau\right\|_{L^\infty({Q(x_0,\sqrt{R})})}^2\\
			\lesssim& \sup_{x_0\in\R^3,0< R\leq T} \left\| t^{-3/2} \int_0^{\frac{t}{2}}\sum_{m=0}^{+\infty}\int_{m\leq \frac{|x-y|}{\sqrt{t}}\leq m+1}\dfrac{1}{(1+m)^4}\left( |\tilde{u}|^2  + |\nabla \tilde{Q}|^2 +\sum_{l=1}^{4}|\tilde{Q}|^l \right)dyd\tau \right\|_{L^\infty(Q(x_0,\sqrt{R}))}^2\\
			& + \sup_{x_0\in\R^3,0< R\leq T} \left(  \int_0^{\frac{t}{2}} \dfrac{\sqrt{R}}{\sqrt{t-\tau}}\cdot\dfrac{1}{\sqrt{t-\tau}\sqrt{\tau}}d\tau \cdot\|\tilde{Q}\|_{\mathbb{X}_T}^2 \right)^2\\
			\lesssim & \left( \|\left(\tilde{Q},\tilde{u}\right)\|_{\mathbb{E}_T}^2 + t\sum_{l=1}^4\|\tilde{Q}\|_{L^\infty}^l \right)^2
		\end{align*}
		and
		\begin{align*}
			&\sup_{x_0\in\R^3,0< R\leq T} R \left\|\nabla\int_\frac{t}{2}^t e^{(t-\tau)\Delta}\mathbb{P}(1-\chi) F_2(\tilde{Q},\tilde{u})d\tau\right\|_{L^\infty({Q(x_0,\sqrt{R})})}^2\\
			\lesssim& \sup_{x_0\in\R^3,0< R\leq T} R \left\| \int_\frac{t}{2}^t\int_{\R^3} \dfrac{1}{(\sqrt{t-\tau}+|x-y|)^4} \left(|\tilde{u}|^2 +|\nabla\tilde{Q}|^2+\sum_{l=1}^4|\tilde{Q}|^l \right)dy d\tau\right\|_{L^\infty({Q(x_0,\sqrt{R})})}^2\\
			& + \sup_{x_0\in\R^3,0< R\leq T} R \left\| \int_\frac{t}{2}^t \int_{\R^3} \dfrac{1}{(\sqrt{t-\tau}+|x-y|)^5} |\nabla\tilde{Q}||\tilde{Q}| dy d\tau\right\|_{L^\infty({Q(x_0,\sqrt{R})})}^2\\
			\lesssim& \sup_{x_0\in\R^3,0< R\leq T} \left(  \int_{\frac{t}{2}}^t \dfrac{\sqrt{R}}{\sqrt{\tau}}\cdot\dfrac{1}{\sqrt{t-\tau}\sqrt{\tau}}d\tau\|\left(\tilde{Q},\tilde{u}\right)\|_{\mathbb{E}_T}^2 + t\sum_{l=1}^4\|\tilde{Q}\|_{L^\infty}^l \right)^2\\
			& + \sup_{x_0\in\R^3,0< R\leq T} \left(  \int_{\frac{t}{2}}^t \dfrac{1}{\sqrt{t-\tau}^{3/2}}\cdot\dfrac{\sqrt{R}}{\sqrt{\tau}}\cdot\dfrac{\tau^{3/4}\|\tilde{Q}\|_{\dot{C}^{1/2}}\|\nabla\tilde{Q}\|_{L^\infty}}{\tau^{1/4}} d\tau  \right)^2\\
			\lesssim & \left(\|\left(\tilde{Q},\tilde{u}\right)\|_{\mathbb{E}_T}^2 + t\sum_{l=1}^4\|\tilde{Q}\|_{L^\infty}^l\right)^2.
		\end{align*}
	\end{itemize}
	%According to the estimates for $(i)$, $(ii)$, we get
	Hence, we obtain that
	\begin{equation}\label{eq3.18}
		I_{22}\lesssim \|\left(\tilde{Q},\tilde{u}\right)\|_{\mathbb{E}_T}^2 + t\sum_{l=1}^4\|\tilde{Q}\|_{L^\infty}^l.
	\end{equation}
	Now, it follows from \eqref{eq3.9}, \eqref{eq3.16} and \eqref{eq3.18} that
	\begin{equation}\label{eq3.20}
		\mathcal{I}_u \lesssim \|\left(\tilde{Q},\tilde{u}\right)\|_{\mathbb{E}_T}^2 + T\sum_{l=1}^4\|\tilde{Q}\|_{L^\infty}^l.
	\end{equation}
	\textbf{Step~2}. In this step, we deal with $\mathcal{I}_Q$. Similar to the previous argument for $\mathcal{I}_u$, we need to check that
	\begin{equation}\label{eq3.21}
		\begin{split}
			\left\| \int_0^t e^{(t-\tau)\Delta} F_1(\tilde{Q},\tilde{u})d\tau \right\|_{L^\infty}+ t^{\frac{1}{4}}\left\| \int_0^t e^{(t-\tau)\Delta} F_1(\tilde{Q},\tilde{u})d\tau \right\|_{\dot{C}^{1/2}}+t^{\frac{1}{2}}\left\| \nabla\int_0^t e^{(t-\tau)\Delta} F_1(\tilde{Q},\tilde{u})d\tau \right\|_{L^\infty}&\\
			+ \sup_{x_0\in\R^3}\sup_{0<R\leq T}\left(R^{-3/2}\int_0^R\int_{B(x_0,\sqrt{R})}\left(\nabla\int_0^t e^{(t-\tau)\Delta} F_1(\tilde{Q},\tilde{u})d\tau \right)^2 dy dt\right)^{\frac{1}{2}}\lesssim \epsilon_0.
		\end{split}
	\end{equation}
	Indeed, it is easy to see that the second term on the left hand side of \eqref{eq3.21} can be demonstrated through interpolation as shown in \eqref{eq2.9}. Thus we only need to estimate the remainder three terms, respectively.
%	\[
%	\begin{split}
%		t^{\frac{1}{4}}\|\tilde{Q}\|_{\dot{C}^{1/2}}=& t^{\frac{1}{4}}\sup_{x\ne y}\dfrac{\left| \tilde{Q}(x)- \tilde{Q}(y) \right|}{|x-y|^{\frac{1}{2}}}\\
%		=& t^{\frac{1}{4}}\sup_{x\ne y}\left(\dfrac{\left| \tilde{Q}(x)- \tilde{Q}(y) \right|}{|x-y|}\right)^{\frac{1}{2}}\cdot \left| \tilde{Q}(x)- \tilde{Q}(y) \right|^{\frac{1}{2}}\\
%		\lesssim &t^{\frac{1}{4}}\| \nabla\tilde{Q}\|_{L^\infty}^{\frac{1}{2}} \|\tilde{Q}\|_{L^\infty}^{\frac{1}{2}} \\
%		\lesssim &\|\left(\tilde{Q},\tilde{u}\right)\|_{X_T}^2 + t\sum_{l=1}^3\|\tilde{Q}\|_{L^\infty}^l.
%	\end{split}
%	\]

	For the first one, we again separate it by the following two parts:
	\[
	\left\| \int_0^t e^{(t-\tau)\Delta}F_1(\tilde{Q},\tilde{u})d\tau\right\|_{L^\infty}=\left\| \int_0^{\frac{t}{2}} +\int_{\frac{t}{2}}^{t}\right\|_{L^\infty} \overset{\Delta}{=} J_{11}+J_{12}.
	\]
	Recall that
	\begin{align*}
		2\left(\tilde{Q}\Omega(\tilde{u})\right)_{\alpha\beta}=& \sum_{j} \partial_\beta\left(Q_{\alpha j} u_j\right) - \partial_j\left(Q_{\alpha j} u_\beta\right)- \partial_\beta Q_{\alpha j} u_j   + \partial_j Q_{\alpha j} u_\beta,\\
		2\left(\Omega(\tilde{u})\tilde{Q}\right)_{\alpha\beta}=& \sum_{j} \partial_j\left(u_\alpha Q_{j\beta} \right) - \partial_\alpha\left(u_jQ_{j\beta} \right)- u_\alpha\partial_\beta Q_{j\beta}    + u_j\partial_\alpha Q_{j\beta} ,\\
		2|\tilde{Q}|D(\tilde{u})=& \nabla\left(|\tilde{Q}|(\tilde{u}+\tilde{u}^T)\right) - \nabla|\tilde{Q}|\otimes(\tilde{u}+\tilde{u}^T),
	\end{align*}
	we then have
	\begin{align*}
		J_{11}\lesssim& \left\|  \int_0^{\frac{t}{2}}\int_{\R^3}\dfrac{e^{-\frac{(x-y)^2}{4(t-\tau)}}}{\sqrt{t-\tau}^3}\left( |\tilde{u}|^2  + |\nabla \tilde{Q}|^2 +\sum_{l=1}^{3}|\tilde{Q}|^l \right)dyd\tau \right\|_{L^\infty} + \left\|  \int_0^{\frac{t}{2}} \nabla e^{(t-\tau)\Delta} |\tilde{Q}||\tilde{u}| d\tau \right\|_{L^\infty}\\
		\lesssim &\left\| \dfrac{1}{\sqrt{t}^3} \sum_{m=0}^{+\infty}\int_0^{\frac{t}{2}}\int_{m\leq \frac{|x-y|}{\sqrt{t-\tau}}\leq m+1}e^{-m^2}\left( |\tilde{u}|^2  + |\nabla \tilde{Q}|^2 +\sum_{l=1}^{3}|\tilde{Q}|^l \right)dyd\tau \right\|_{L^\infty}\\
		& +\left\|  \int_0^{\frac{t}{2}}\dfrac{1}{\sqrt{t-\tau}}\cdot\frac{\sqrt{\tau}}{\sqrt{\tau}}\|\tilde{u}\|_{L^\infty}\| \tilde{Q}\|_{L^\infty}d\tau \right\|_{L^\infty}\\
		\lesssim &\left\|  \sum_{m=0}^{+\infty}m^2e^{-m^2}t^{-\frac{3}{2}}\int_0^{\frac{t}{2}}\int_{B(x,\sqrt{t})}\left( |\tilde{u}|^2  + |\nabla \tilde{Q}|^2 +\sum_{l=1}^{3}|\tilde{Q}|^l \right)dyd\tau \right\|_{L^\infty} +\|\left(\tilde{Q},\tilde{u}\right)\|_{\mathbb{E}_T}^2   \int_0^{\frac{1}{2}}\dfrac{1}{\sqrt{1-s}\sqrt{s}} ds  \\
		\lesssim & \|\left(\tilde{Q},\tilde{u}\right)\|_{\mathbb{E}_T}^2 + t\sum_{l=1}^3\|\tilde{Q}\|_{L^\infty}^l
	\end{align*}
	and
	\begin{align*}
		J_{12}\lesssim& \left\|  \int_{\frac{t}{2}}^t\int_{\R^3}\dfrac{e^{-\frac{(x-y)^2}{4(t-\tau)}}}{\sqrt{t-\tau}^3}\left( |\tilde{u}|^2  + |\nabla \tilde{Q}|^2 +\sum_{l=1}^{3}|\tilde{Q}|^l \right)dyd\tau \right\|_{L^\infty} + \left\|  \int_{\frac{t}{2}}^t \nabla\mathbb{P}e^{(t-\tau)\Delta} |\tilde{Q}||\tilde{u}| d\tau \right\|_{L^\infty}\\
		\lesssim & \left\|  \int_{\frac{t}{2}}^t\int_{\R^3}\dfrac{e^{-\frac{(x-y)^2}{4(t-\tau)}}}{\sqrt{t-\tau}^3}\cdot\frac{\tau}{\tau}\left( \|\tilde{u}\|_{L^\infty}^2  + \|\nabla \tilde{Q}\|_{L^\infty}^2  \right)dyd\tau \right\|_{L^\infty} + t\sum_{l=1}^3\|\tilde{Q}\|_{L^\infty}^l  +\left\|  \int_{\frac{t}{2}}^t \frac{\|\left(\tilde{Q},\tilde{u}\right)\|_{\mathbb{E}_T}^2}{\sqrt{t-\tau}\sqrt{\tau}}d\tau \right\|_{L^\infty}\\
		\lesssim &t\sum_{l=1}^3\|\tilde{Q}\|_{L^\infty}^l  +\|\left(\tilde{Q},\tilde{u}\right)\|_{\mathbb{E}_T}^2  \left(\left\|  \int_{\frac{t}{2}}^t\frac{1}{\tau}d\tau\int_{\R^3}\dfrac{e^{-\frac{(x-y)^2}{4(t-\tau)}}}{\sqrt{t-\tau}^3}dy \right\|_{L^\infty}+  \int_{\frac{1}{2}}^1\dfrac{1}{\sqrt{1-s}\sqrt{s}} ds\right)  \\
		\lesssim & \|\left(\tilde{Q},\tilde{u}\right)\|_{\mathbb{E}_T}^2 + t\sum_{l=1}^3\|\tilde{Q}\|_{L^\infty}^l.
	\end{align*}
	Similarly, we can regard the third term in \eqref{eq3.21} as
	\begin{align}
		t^{\frac{1}{2}}\left\| \nabla \int_0^t e^{(t-\tau)\Delta} F_1(\tilde{Q},\tilde{u})d\tau \right\|_{L^\infty}=t^{\frac{1}{2}}\left\| \nabla\left(\int_0^{\frac{t}{2}} +\int_{\frac{t}{2}}^{t}\right)\right\|_{L^\infty} \overset{\Delta}{=} J_{21}+J_{22},
	\end{align}
	and thus
	\begin{align*}
		J_{21}\lesssim&  t^{\frac{1}{2}}\left\|  \int_0^{\frac{t}{2}}\int_{\R^3}\dfrac{1}{\left(\sqrt{t-\tau} +|x-y| \right)^{4}}\left( |\tilde{u}|^2 + |\nabla \tilde{Q}|^2+ \sum_{l=1}^{3}|\tilde{Q}|^l \right)dyd\tau \right\|_{L^\infty}\\
		&+t^{\frac{1}{2}}\left\|  \int_0^{\frac{t}{2}}\left|\int_{\R^3}\dfrac{1}{\left(\sqrt{t-\tau} +|x-y| \right)^{5}}dy\right|\|\tilde{u}\|_{L^\infty}\| \tilde{Q}\|_{L^\infty}d\tau \right\|_{L^\infty}\\
		\lesssim& \left\| t^{-\frac{3}{2}} \int_0^{\frac{t}{2}}\int_{\R^3}\dfrac{1}{\left(1 +\frac{|x-y|}{\sqrt{t-\tau}} \right)^{4}}\left( |\tilde{u}|^2 + |\nabla \tilde{Q}|^2 + \sum_{l=1}^{3}|\tilde{Q}|^l \right)dyd\tau \right\|_{L^\infty} +t^{\frac{1}{2}}\left\|  \int_0^{\frac{t}{2}}\dfrac{1}{t-\tau}\dfrac{\|(\tilde{Q},\tilde{u})\|_{\mathbb{E}_T}^2}{\sqrt{\tau}}d\tau \right\|_{L^\infty}\\
		\lesssim& \left\| t^{-\frac{3}{2}}\sum_{m=0}^{+\infty} \int_0^{\frac{t}{2}}\int_{m\leq \frac{|x-y|}{\sqrt{t-\tau}}\leq m+1}\dfrac{|\tilde{u}|^2 + |\nabla \tilde{Q}|^2+ \sum_{l=1}^{3}|\tilde{Q}|^l}{(1 +m )^{4}} dyd\tau \right\|_{L^\infty}+ \|(\tilde{Q},\tilde{u})\|_{\mathbb{E}_T}^2\left\|  \int_0^{\frac{t}{2}}\dfrac{1}{t-\tau}\dfrac{\sqrt{t}}{\sqrt{\tau}}d\tau \right\|_{L^\infty}\\
		\lesssim& \left\| \sum_{m=0}^{+\infty}\dfrac{m^2}{(1 +m )^{4}} t^{-\frac{3}{2}}\int_0^{\frac{t}{2}}\int_{B(x,\sqrt{t})}\left(|\tilde{u}|^2 + |\nabla \tilde{Q}|^2 + \sum_{l=1}^{3}|\tilde{Q}|^l\right) dyd\tau \right\|_{L^\infty}+ \|(\tilde{Q},\tilde{u})\|_{\mathbb{E}_T}^2\int_0^{\frac{1}{2}}\dfrac{ds}{\left(1-s\right)\sqrt{s}} \\
		\lesssim&  \|\left(\tilde{Q},\tilde{u}\right)\|_{\mathbb{E}_T}^2  + t\sum_{l=1}^3\| \tilde{Q}\|_{L^\infty}^l,\\
		J_{22} \lesssim& t^{\frac{1}{2}}\left\|  \int_{\frac{t}{2}}^t\left|\int_{\R^3}\dfrac{1}{\left(\sqrt{t-\tau} +|x-y| \right)^{4}}dy\right|\left(\|\tilde{u}\|_{L^\infty}^2+ \|\nabla \tilde{Q}\|_{L^\infty}^2 +\sum_{l=1}^3 \|\tilde{Q}\|_{L^\infty}^l   \right)d\tau \right\|_{L^\infty}\\
		&+ t^{\frac{1}{2}}\left\|  \int_{\frac{t}{2}}^t \left|\int_{\R^3}\dfrac{1}{\left(\sqrt{t-\tau} +|x-y| \right)^{5}}dy\right| |\tilde{u}| |\tilde{Q}| d\tau \right\|_{L^\infty}\\
		\lesssim&  t^{\frac{1}{2}}\left\|  \int_{\frac{t}{2}}^t\dfrac{1}{\sqrt{t-\tau}}\dfrac{\|\left(\tilde{Q},\tilde{u}\right)\|_{\mathbb{E}_T}^2}{\tau}d\tau \right\|_{L^\infty}+ t\sum_{l=1}^3\|\tilde{Q}\|_{L^\infty}^l + t^{\frac{1}{2}}\left\|  \int_{\frac{t}{2}}^t\dfrac{1}{\sqrt{t-\tau}^{3/2}}\|\tilde{u}\|_{L^\infty}\| \tilde{Q}\|_{\dot{C}^{1/2}}d\tau \right\|_{L^\infty}\\
		\lesssim&  \|\left(\tilde{Q},\tilde{u}\right)\|_{\mathbb{E}_T}^2\cdot  \int_{\frac{t}{2}}^t\dfrac{t^{-1}}{\sqrt{1-\frac{\tau}{t}}\cdot\frac{\tau}{t}} d\tau + t\sum_{l=1}^3\|\tilde{Q}\|_{L^\infty}^l + t^{\frac{1}{2}}\left\|  \int_{\frac{t}{2}}^t\dfrac{1}{\sqrt{t-\tau}^{3/2}}\dfrac{\|\left(\tilde{Q},\tilde{u}\right)\|_{\mathbb{E}_T}^2}{\tau^{3/4}}d\tau \right\|_{L^\infty}\\
		\lesssim&  \|\left(\tilde{Q},\tilde{u}\right)\|_{\mathbb{E}_T}^2 \left(\int_{\frac{1}{2}}^1\dfrac{1}{ \sqrt{1-s} \cdot s}+ \dfrac{1}{ \sqrt{1-s}^{3/2} \cdot s^{3/4}}ds\right)+ t\sum_{l=1}^3\|\tilde{Q}\|_{L^\infty}^l \\
		\lesssim&  \|\left(\tilde{Q},\tilde{u}\right)\|_{\mathbb{E}_T}^2 + t\sum_{l=1}^3
		\|\tilde{Q}\|_{L^\infty}^l.
	\end{align*}
	Finally, we deal with the following term via a similar argument as what we did in the first step,
	\[
	\sup_{x_0\in\R^3}\sup_{0<R\leq T}\left(R^{-3/2}\int_0^R\int_{B(x_0,\sqrt{R})}\left(\nabla\int_0^t e^{(t-\tau)\Delta} F_1(\tilde{Q},\tilde{u})d\tau \right)^2 dy dt\right)^{\frac{1}{2}}.
	\]
	Indeed, we shall estimate
	\begin{equation}
		\begin{split}
			%J_3:=&\sup_{x_0\in\R^3,0<R\leq T}\left(R^{-3/2}\int_0^R\int_{B(x_0,\sqrt{R})}\left(\nabla\int_0^t e^{(t-\tau)\Delta} F_1(\tilde{Q},\tilde{u})d\tau \right)^2 dy dt\right)^{\frac{1}{2}}\\
			%=
			&\sup_{x_0\in\R^3,0<R\leq T}\left(R^{-3/2}\int_0^R\int_{B(x_0,\sqrt{R})}\left(\nabla\int_0^t e^{(t-\tau)\Delta} \chi F_1(\tilde{Q},\tilde{u})d\tau \right)^2 dy dt\right)^{\frac{1}{2}}\\
			&+ \sup_{x_0\in\R^3,0<R\leq T}\left(R^{-3/2}\int_0^R\int_{B(x_0,\sqrt{R})}\left(\nabla\int_0^t e^{(t-\tau)\Delta} (1-\chi) F_1(\tilde{Q},\tilde{u})d\tau \right)^2 dy dt\right)^{\frac{1}{2}}
			\overset{\Delta}{=} I_{31} + I_{32}.
		\end{split}
	\end{equation}
	%To estimate $J_{31}$,
	Similarly, we will study the Cauchy problem of the following heat equation
	\begin{equation}
		\begin{cases}
			V_t-\Delta V=\chi F_1(\tilde{Q},\tilde{u}),\\
			V\mid_{t=0}=0.
		\end{cases}
	\end{equation}
	Then, by exploiting an energy argument, we infer that
	\begin{align}\label{eq3.25}
		J_{31}\lesssim \sup_{x_0\in\R^3,0<R\leq T}\left(\|V\|_{L^\infty(Q(x_0,5\sqrt{R}))}\cdot R^{-3/2}\int_0^R \int_{\R^3}|\chi F_1(\tilde{Q},\tilde{u})| dyd\tau \right)^{\frac{1}{2}}.
	\end{align}
	Moreover, by the Duhamel’s principle, we also have
	\[
	\|V\|_{L^\infty(Q(x_0,5\sqrt{R}))}= \left\| \int_0^{\frac{t}{2}} e^{(t-\tau)\Delta}\chi F_1(\tilde{Q},\tilde{u})  d\tau + \int_{\frac{t}{2}}^t e^{(t-\tau)\Delta}\chi F_1(\tilde{Q},\tilde{u}) d\tau \right\|_{L^\infty(Q(x_0,5\sqrt{R}))}.
	\]
	Observe that
	\begin{align}\label{eq3.26}
		&\left\| \int_0^{\frac{t}{2}} e^{(t-\tau)\Delta}\chi F_1(\tilde{Q},\tilde{u})  d\tau \right\|_{L^\infty(Q(x_0,5\sqrt{R}))}\nonumber\\
		\lesssim& \left\| \int_0^{\frac{t}{2}} e^{(t-\tau)\Delta}  \left(|\tilde{u}|^2 +|\nabla\tilde{Q}|^2+\sum_{l=1}^3|\tilde{Q}|^l \right)  d\tau \right\|_{L^\infty(Q(x_0,5\sqrt{R}))} + \left\| \int_0^{\frac{t}{2}} \nabla e^{(t-\tau)\Delta} |\tilde{u}||\tilde{Q}| d\tau \right\|_{L^\infty(Q(x_0,5\sqrt{R}))}\nonumber\\
		\lesssim& \left\| \dfrac{1}{\sqrt{t}^3} \sum_{m=0}^{+\infty}\int_0^{\frac{t}{2}}\int_{m\leq \frac{|x-y|}{\sqrt{t-\tau}}\leq m+1}e^{-m^2}\left( |\tilde{u}|^2  + |\nabla \tilde{Q}|^2 +\sum_{l=1}^{3}|\tilde{Q}|^l \right)dyd\tau \right\|_{L^\infty(Q(x_0,5\sqrt{R}))}\\
		& +\left\|  \int_0^{\frac{t}{2}}\dfrac{1}{\sqrt{t-\tau}}\cdot\frac{\sqrt{\tau}}{\sqrt{\tau}}\|\tilde{u}\|_{L^\infty}\| \tilde{Q}\|_{L^\infty}d\tau \right\|_{L^\infty(Q(x_0,5\sqrt{R}))}\nonumber\\
		\lesssim & \|\left(\tilde{Q},\tilde{u}\right)\|_{\mathbb{E}_T}^2 + t\sum_{l=1}^3\|\tilde{Q}\|_{L^\infty}^l,\nonumber
	\end{align}
	\begin{equation}
		\begin{split}
			&\left\| \int_{\frac{t}{2}}^t e^{(t-\tau)\Delta}\chi F_1(\tilde{Q},\tilde{u})  d\tau \right\|_{L^\infty(Q(x_0,5\sqrt{R}))}\\
			\lesssim & t\sum_{l=1}^3\|\tilde{Q}\|_{L^\infty}^l  +\left\|  \int_{\frac{t}{2}}^t\int_{\R^3}\dfrac{e^{-\frac{(x-y)^2}{4(t-\tau)}}}{\sqrt{t-\tau}^3}\cdot\frac{\tau}{\tau}\left( \|\tilde{u}\|_{L^\infty}^2  + \|\nabla \tilde{Q}\|_{L^\infty}^2  \right)dyd\tau \right\|_{L^\infty} + \left\|  \int_{\frac{t}{2}}^t \frac{\|\left(\tilde{Q},\tilde{u}\right)\|_{\mathbb{E}_T}^2}{\sqrt{t-\tau}\sqrt{\tau}}d\tau \right\|_{L^\infty}\\
			\lesssim &t\sum_{l=1}^3\|\tilde{Q}\|_{L^\infty}^l  +\|\left(\tilde{Q},\tilde{u}\right)\|_{\mathbb{E}_T}^2  \left(\left\|  \int_{\frac{t}{2}}^t\frac{1}{\tau}d\tau\int_{\R^3}\dfrac{e^{-\frac{(x-y)^2}{4(t-\tau)}}}{\sqrt{t-\tau}^3}dy \right\|_{L^\infty}+  \int_{\frac{1}{2}}^1\dfrac{1}{\sqrt{1-s}\sqrt{s}} ds\right)  \\
			\lesssim & \|\left(\tilde{Q},\tilde{u}\right)\|_{\mathbb{E}_T}^2 + t\sum_{l=1}^3\|\tilde{Q}\|_{L^\infty}^l,
		\end{split}
	\end{equation}
	and
	\begin{equation}\label{eq3.28}
		\begin{split}
			R^{-3/2}\int_0^R \int_{\R^3}&|\chi F_1(\tilde{Q},\tilde{u})| dyd\tau\\
			\lesssim &R^{-3/2}\int_0^R \int_{B(x_0,5\sqrt{R})}\left(|\tilde{u}|^2 +|\nabla\tilde{Q}|^2+\sum_{l=1}^3|\tilde{Q}|^l \right) dyd\tau\\
			& + R^{-3/2} \int_0^R  \int_{B(x_0,5\sqrt{R})}\frac{1}{\sqrt{R}}\nabla_y \chi\left(\frac{y}{\sqrt{R}}\right) |\tilde{u}||\tilde{Q}| dyd\tau\\
			\lesssim & \|\left(\tilde{Q},\tilde{u}\right)\|_{\mathbb{E}_T}^2 + R\sum_{l=1}^3\|\tilde{Q}\|_{L^\infty}^l+  R^{-2}\int_0^R\int_{B(x_0,5\sqrt{R})}\frac{1}{\sqrt{\tau}}dyd\tau\cdot\|\left(\tilde{Q},\tilde{u}\right)\|_{\mathbb{E}_T}^2\\
			\lesssim & \|\left(\tilde{Q},\tilde{u}\right)\|_{\mathbb{E}_T}^2 + T\sum_{l=1}^3\|\tilde{Q}\|_{L^\infty}^l.
		\end{split}
	\end{equation}
	Thus, substituting \eqref{eq3.26}-\eqref{eq3.28} into \eqref{eq3.25}, we get
	\begin{align*}
		J_{31}\lesssim \|\left(\tilde{Q},\tilde{u}\right)\|_{\mathbb{E}_T}^2 + T\sum_{l=1}^3\|\tilde{Q}\|_{L^\infty}^l.
	\end{align*}
	For $J_{32}$, we again estimate
	$
	\sup_{x_0\in\R^3,0<R\leq T} R \left\|\nabla\int_0^t e^{(t-\tau)\Delta} (1-\chi) F_1(\tilde{Q},\tilde{u})d\tau\right\|_{L^\infty({Q(x_0,\sqrt{R})})}^2.
	$
	Similarly, we distinguish two cases:
	\begin{itemize}
		\item [(i)] When $0\leq \tau\leq t\leq (\sqrt{R}/2)^2$, we have
		\begin{align*}
			&\sup_{x_0\in\R^3,0<R\leq T} R \left\|\nabla\int_0^t e^{(t-\tau)\Delta} (1-\chi) F_1(\tilde{Q},\tilde{u})d\tau\right\|_{L^\infty({Q(x_0,\sqrt{R})})}^2\\
			\lesssim& \sup_{x_0\in\R^3,0<R\leq T} R \left\| \int_0^t\int_{\R^3} \dfrac{1-\chi(y/\sqrt{R})}{(\sqrt{t-\tau}+|x-y|)^4} \left(|\tilde{u}|^2 +|\nabla\tilde{Q}|^2+\sum_{l=1}^3|\tilde{Q}|^l \right)dy d\tau\right\|_{L^\infty({Q(x_0,\sqrt{R})})}^2\\
			& + \sup_{x_0\in\R^3,0<R\leq T} R \left\| \int_0^t\int_{\R^3} \dfrac{1}{(\sqrt{t-\tau}+|x-y|)^5} |\tilde{u}||\tilde{Q}| dy d\tau\right\|_{L^\infty({Q(x_0,\sqrt{R})})}^2\\
			\lesssim& \sup_{x_0\in\R^3,0<R\leq T} R\left\|  \int_0^{t}\sum_{m=1}^{+\infty}\int_{m\sqrt{R}\leq |x-y|\leq (m+1)\sqrt{R}}\dfrac{1}{(m\sqrt{R})^4}\left( |\tilde{u}|^2  + |\nabla \tilde{Q}|^2 +\sum_{l=1}^{3}|\tilde{Q}|^l \right)dyd\tau \right\|_{L^\infty(Q(x_0,\sqrt{R}))}^2\\
			& + \sup_{x_0\in\R^3,0<R\leq T} \left\|  \int_0^{t} \dfrac{1}{\sqrt{t-\tau}}\cdot\dfrac{\sqrt{\tau}}{\sqrt{\tau}}\| \tilde{u}\|_{L^\infty}  \|\tilde{Q}\|_{L^\infty}dyd\tau \right\|_{L^\infty(Q(x_0,\sqrt{R}))}^2\\
			\lesssim& \sup_{x_0\in\R^3,0<R\leq T} \left( \sup_{x\in B(x_0,\sqrt{R})} R^{-3/2} \int_0^{R} \int_{B(x,\sqrt{R})}( |\tilde{u}|^2  + |\nabla \tilde{Q}|^2 )dyd\tau + t\sum_{l=1}^3\|\tilde{Q}\|_{L^\infty}^l \right)^2+ \|\left(\tilde{Q},\tilde{u}\right)\|_{\mathbb{E}_T}^4\\
			\lesssim & \left(\|\left(\tilde{Q},\tilde{u}\right)\|_{\mathbb{E}_T}^2 + t\sum_{l=1}^3\|\tilde{Q}\|_{L^\infty}^l\right)^2.
		\end{align*}
		\item [(ii)] When $(\sqrt{R}/2)^2\leq t\leq R$, we have
		\begin{align*}
			&\sup_{x_0\in\R^3,0<R\leq T} R \left\|\nabla\int_0^\frac{t}{2} e^{(t-\tau)\Delta} (1-\chi) F_1(\tilde{Q},\tilde{u})d\tau\right\|_{L^\infty({Q(x_0,\sqrt{R})})}^2\\
			\lesssim& \sup_{x_0\in\R^3,0<R\leq T} R \left\| \int_0^\frac{t}{2}\int_{\R^3} \dfrac{1}{(\sqrt{t-\tau}+|x-y|)^4} \left(|\tilde{u}|^2 +|\nabla\tilde{Q}|^2+\sum_{l=1}^3|\tilde{Q}|^l \right)dy d\tau\right\|_{L^\infty({Q(x_0,\sqrt{R})})}^2\\
			& + \sup_{x_0\in\R^3,0<R\leq T} R \left\| \int_0^\frac{t}{2}\int_{\R^3} \dfrac{1}{(\sqrt{t-\tau}+|x-y|)^5} |\tilde{u}||\tilde{Q}| dy d\tau\right\|_{L^\infty({Q(x_0,\sqrt{R})})}^2\\
			\lesssim& \sup_{x_0\in\R^3,0<R\leq T} \left\| t^{-3/2} \int_0^{\frac{t}{2}}\sum_{m=0}^{+\infty}\int_{m\leq \frac{|x-y|}{\sqrt{t}}\leq m+1}\dfrac{1}{(1+m)^4}\left( |\tilde{u}|^2  + |\nabla \tilde{Q}|^2 +\sum_{l=1}^{3}|\tilde{Q}|^l \right)dyd\tau \right\|_{L^\infty(Q(x_0,\sqrt{R}))}^2\\
			& + \sup_{x_0\in\R^3,0<R\leq T} \left(  \int_0^{\frac{t}{2}} \dfrac{\sqrt{R}}{\sqrt{t-\tau}}\cdot\dfrac{1}{\sqrt{t-\tau}\sqrt{\tau}}d\tau \|\left(\tilde{Q},\tilde{u}\right)\|_{\mathbb{E}_T}^2 \right)^2\\
			\lesssim & \left(\|\left(\tilde{Q},\tilde{u}\right)\|_{\mathbb{E}_T}^2 + t\sum_{l=1}^3\|\tilde{Q}\|_{L^\infty}^l\right)^2,
		\end{align*}
		and
		\begin{align*}
			&\sup_{x_0\in\R^3,0<R\leq T} R \left\|\nabla\int_\frac{t}{2}^t e^{(t-\tau)\Delta} (1-\chi) F_1(\tilde{Q},\tilde{u})d\tau\right\|_{L^\infty({Q(x_0,\sqrt{R})})}^2\\
			\lesssim& \sup_{x_0\in\R^3,0<R\leq T} R \left\| \int_\frac{t}{2}^t\int_{\R^3} \dfrac{1}{(\sqrt{t-\tau}+|x-y|)^4} \left(|\tilde{u}|^2 +|\nabla\tilde{Q}|^2+\sum_{l=1}^3|\tilde{Q}|^l \right)dy d\tau\right\|_{L^\infty({Q(x_0,\sqrt{R})})}^2\\
			& + \sup_{x_0\in\R^3,0<R\leq T} R \left\| \int_\frac{t}{2}^t \int_{\R^3} \dfrac{1}{(\sqrt{t-\tau}+|x-y|)^5} |\tilde{u}||\tilde{Q}| dy d\tau\right\|_{L^\infty({Q(x_0,\sqrt{R})})}^2\\
			\lesssim& \sup_{x_0\in\R^3,0<R\leq T} \left(  \int_{\frac{t}{2}}^t \dfrac{\sqrt{R}}{\sqrt{\tau}}\cdot\dfrac{1}{\sqrt{t-\tau}\sqrt{\tau}}d\tau \|\left(\tilde{Q},\tilde{u}\right)\|_{\mathbb{E}_T}^2 + t\sum_{l=1}^3\|\tilde{Q}\|_{L^\infty}^l \right)^2\\
			& + \sup_{x_0\in\R^3,0<R\leq T} \left(  \int_{\frac{t}{2}}^t \dfrac{1}{\sqrt{t-\tau}^{3/2}}\cdot\dfrac{\sqrt{R}}{\sqrt{\tau}}\cdot \dfrac{\tau^{3/4}\|\tilde{Q}\|_{\dot{C}^{1/2}}\|\tilde{u}\|_{L^\infty}}{\tau^{1/4}}d\tau \right)^2\\
			\lesssim & \left(\|\left(\tilde{Q},\tilde{u}\right)\|_{\mathbb{E}_T}^2 + t\sum_{l=1}^3\|\tilde{Q}\|_{L^\infty}^l\right)^2.
		\end{align*}
	\end{itemize}
	Hence, we infer that
	\[
	J_{32}\lesssim \|\left(\tilde{Q},\tilde{u}\right)\|_{\mathbb{E}_T}^2 + T\sum_{l=1}^3\|\tilde{Q}\|_{L^\infty}^l
	\]
	and thus
	\begin{equation}\label{eq3.29}
		\mathcal{I}_Q \lesssim \|\left(\tilde{Q},\tilde{u}\right)\|_{\mathbb{E}_T}^2 + T\sum_{l=1}^3\|\tilde{Q}\|_{L^\infty}^l.
	\end{equation}
\textbf{Step~3}. For any $(\tilde{Q},\tilde{u})\in \mathcal{B}_{\epsilon_0}(Q_L,u_L)$, by Proposition~\ref{prop2.1}, one has
\begin{equation}
	\|\left(\tilde{Q},\tilde{u}\right)\|_{\mathbb{E}_T}\leq C\epsilon_0.
\end{equation}
Then, it follows from \eqref{eq3.20} and \eqref{eq3.29} that
	\begin{equation}
		\|Q-Q_L\|_{\mathbb{X}_T} + \|u-u_L\|_{\mathbb{Y}_T}\lesssim \|\left(\tilde{Q},\tilde{u}\right)\|_{\mathbb{E}_T}^2  + T\sum_{l=1}^4\|\tilde{Q}\|_{L^\infty}^l\lesssim \epsilon_0^2+T\sum_{l=1}^4 \epsilon_0^l\lesssim \epsilon_0,
	\end{equation}
	for some fixed $0<T<+\infty$, provided $\epsilon_0>0$ is chosen to be sufficiently small. This completes the proof of Lemma~\ref{lem3.1}.
	%This implies that $S$ maps $(\tilde{Q},\tilde{u})\in X_T$ into $X_T$.
\end{proof}

\begin{lemma}\label{lem3.2}
	There exist $0<\epsilon_1\leq \epsilon_0$ and $\theta\in (0,1)$ such that if
	\[
	\|Q_0\|_{L^\infty}+ \|u_0\|_{{\rm BMO}^{-1}}\leq \epsilon_1,
	\]
	Then $\bm{\mathcal{S}}: \mathcal{B}_{\epsilon_1}(Q_L,u_L) \rightarrow \mathcal{B}_{\epsilon_1}(Q_L,u_L)$ is $\theta$-contractive, that is
	\begin{equation}\label{eq3.32}
		\|S(Q^1,u^1) - S(Q^2,u^2)\|_{\mathbb{E}_T} \leq\theta  \|(Q^1,u^1) - (Q^2,u^2)\|_{\mathbb{E}_T}
	\end{equation}
	for any $(u^1,Q^1), (u^2,Q^2)\in \mathcal{B}_{\epsilon_1}(Q_L,u_L)$.
\end{lemma}
\begin{proof}
	For any $(u^1,Q^1), (u^2,Q^2)\in \mathcal{B}_{\epsilon_1}(Q_L,u_L)$, we rewrite
	\begin{align*}
		(\mathbb{S}_1(\bar{Q},\bar{u}),\mathbb{S}_2(\bar{Q},\bar{u})):=&S(u^1,Q^1)-S(u^2,Q^2)\\
		=&(S_1(u^1,Q^1)-S_1(u^2,Q^2),S_2(u^1,Q^1)-S_2(u^2,Q^2)).
	\end{align*}
	Then, it follows from \eqref{eq3.4} and \eqref{eq3.5} that
	\begin{numcases}{}
		\mathbb{S}_1(\bar{Q},\bar{u})=-\int_0^t e^{(t-\tau)\Delta}\left(F_1(Q^1,u^1)-F_1(Q^2,u^2) \right)d\tau\overset{\Delta}{=}-\int_0^t e^{(t-\tau)\Delta}\mathbb{N}_1(\bar{Q},\bar{u})d\tau,\\
		\mathbb{S}_2(\bar{Q},\bar{u})=-\int_0^t e^{(t-\tau)\Delta}\mathbb{P}\nabla\cdot\left(F_2(Q^1,u^1)-F_2(Q^2,u^2) \right)d\tau\overset{\Delta}{=}-\int_0^t e^{(t-\tau)\Delta}\mathbb{P}\nabla\cdot\mathbb{N}_2(\bar{Q},\bar{u})d\tau,
	\end{numcases}
	where
	\begin{align*}
		&\bar{Q}=Q^1-Q^2,\quad\bar{u}=u^1-u^2,\quad\bar{\Omega}=\Omega^1-\Omega^2\quad\text{and}\quad\bar{D}=D^1-D^2,\\
		&\mathbb{N}_1(\bar{Q},\bar{u})=\bar{u}\cdot\nabla Q^1 +u^2\cdot\nabla\bar{Q} + Q^1\bar{\Omega}- \bar{\Omega}Q^1+ \bar{Q}\Omega^2 - \Omega^2\bar{Q} -\lambda|Q^1|\bar{D}\\
		&\quad\quad\quad\quad\,\,\,-\lambda\left(|Q^1|-|Q^2|\right)D^2-b\Gamma\left[Q^1\bar{Q}+\bar{Q}Q^2-\frac{{\rm Tr}(Q^1\bar{Q}+\bar{Q}Q^2)}{3}\mathbb{I}_3\right]\\
		&\quad\quad\quad\quad\,\,\,+c\Gamma\left[\bar{Q}{\rm Tr}((Q^1)^2)+Q^2{\rm Tr}(Q^1\bar{Q}+\bar{Q}Q^2)\right]+a\Gamma\bar{Q},\\
		&\mathbb{N}_2(\bar{Q},\bar{u})=u^1\otimes\bar{u}+\bar{u}\otimes u^2+ \nabla Q^1\odot\nabla\bar{Q} + \nabla\bar{Q}\odot\nabla Q^2+ \Delta\bar{Q}Q^1-Q^1\Delta\bar{Q}+\Delta Q^2\bar{Q}-\bar{Q}\Delta Q^2\\
		&\quad\quad\quad\quad\,\,\,\,+ \lambda|Q^1|\Delta \bar{Q} + \lambda(|Q^1|-|Q^2|)\Delta Q^2-a\lambda|Q^1|\bar{Q}-a\lambda(|Q^1|-|Q^2|)Q^2+ b\lambda(|Q^1|-|Q^2|)(Q^2)^2\\
		&\quad\quad\quad\quad\,\,\,\,+b\lambda\left[|Q^1|(Q^1\bar{Q}+\bar{Q}Q^2)-\frac{(|Q^1|-|Q^2|){\rm Tr}\left((Q^2)^2\right)+|Q^1|{\rm Tr}(Q^1\bar{Q}+\bar{Q}Q^2)}{3}\mathbb{I}_3\right]-\kappa\bar{Q}\\
		&\quad\quad\quad\quad\,\,\,\,-c\lambda\left[|Q^1|\bar{Q}{\rm Tr}((Q^2)^2)+ (|Q^1|-|Q^2|)Q^2{\rm Tr}((Q^2)^2)+|Q^1|Q^1{\rm Tr}(Q^1\bar{Q}+\bar{Q}Q^2)\right].
	\end{align*}
	Now, we turn to show \eqref{eq3.32}.
	%$\|\mathbb{S}_2(\bar{Q},\bar{u}))\|_{\mathbb{Y}_T}$.
	Building on the proof of Lemma~\ref{lem3.1}, we similarly deal with
	\begin{align}\label{eq3.35}
		\|\mathbb{S}_2(\bar{Q},\bar{u})\|_{\mathbb{Y}_T}=&t^{\frac{1}{2}}\|\mathbb{S}_2(\bar{Q},\bar{u})\|_{L^\infty}+ \sup_{x_0\in\R^3,0<R\leq T}\left(R^{-3/2}\int_0^R\int_{B(x_0,\sqrt{R})}| \mathbb{S}_2(\bar{Q},\bar{u})(y,t)|^2 dy dt\right)^{\frac{1}{2}}\nonumber\\
		\lesssim& \underbrace{t^{\frac{1}{2}}\left\|\int_0^\frac{t}{2} \nabla e^{(t-\tau)\Delta}\mathbb{N}_2(\bar{Q},\bar{u})d\tau\right\|_{L^\infty}}_{\mathcal{N}_1} + \underbrace{t^{\frac{1}{2}}\left\|\int_\frac{t}{2}^t \nabla e^{(t-\tau)\Delta}\mathbb{N}_2(\bar{Q},\bar{u})d\tau\right\|_{L^\infty}}_{\mathcal{N}_2}\nonumber\\
		& +\sup_{x_0\in\R^3,R\leq T}\Bigg( \underbrace{ R^{-\frac{3}{2}}\int_0^R \int_{\R^3}| \chi \mathbb{N}_2(\bar{Q},\bar{u})| dyd\tau }_{\mathcal{N}_3}\cdot\Bigg(\underbrace{\left\| \int_0^{\frac{t}{2}} e^{(t-\tau)\Delta}\chi \mathbb{N}_2(\bar{Q},\bar{u})  d\tau\right\|_{L^\infty(Q(x_0,5\sqrt{R}))}}_{\mathcal{N}_4} \nonumber\\
		& + \underbrace{\left\| \int_{\frac{t}{2}}^t e^{(t-\tau)\Delta}\chi \mathbb{N}_2(\bar{Q},\bar{u}) d\tau \right\|_{L^\infty(Q(x_0,5\sqrt{R}))}}_{\mathcal{N}_5}\Bigg)\Bigg)^{\frac{1}{2}}\\
		& + \left( \underbrace{\sup_{x_0\in\R^3,0<R\leq T} R \left\|\nabla\int_0^t e^{(t-\tau)\Delta} (1-\chi) \mathbb{N}_2(\bar{Q},\bar{u})d\tau\right\|_{L^\infty({Q(x_0,\sqrt{R})})}^2}_{\mathcal{N}_6} \right)^{\frac{1}{2}}\nonumber 
	\end{align}
	For $\mathcal{N}_i (1\leq i\leq 5)$, we are able to evaluate them as follows:
	\begin{align*}
		\mathcal{N}_1\lesssim& t^{\frac{1}{2}}\left\|  \int_0^{\frac{t}{2}} \nabla e^{(t-\tau)\Delta} \left( \left(|u^1|+ |u^2| \right) |\bar{u}|+ \left(|\nabla Q^1|+ |\nabla Q^2| \right) |\nabla\bar{Q}| +\sum_{l=0}^{3}\left(|Q^1|^l + |Q^2|^l\right)|\bar{Q}| \right)d\tau \right\|_{L^\infty}\\
		& +t^{\frac{1}{2}}\left\|  \int_0^{\frac{t}{2}} \nabla^2 e^{(t-\tau)\Delta} \left(|\nabla\bar{Q}|| Q^1|+ |\nabla Q^2|| \bar{Q}| \right)d\tau \right\|_{L^\infty}\\
		\lesssim&t\sum_{l=0}^3\left(\| Q^1\|_{\mathbb{X}_T}^l+ \| Q^2\|_{\mathbb{X}_T}^l\right)\|\bar{Q}\|_{\mathbb{X}_T}+ t^{\frac{1}{2}} \left\| \int_0^{\frac{t}{2}} \dfrac{d\tau}{(t-\tau)\sqrt{\tau}} \left(\| Q^1\|_{\mathbb{X}_T}+ \| Q^2\|_{\mathbb{X}_T}\right)\|\bar{Q}\|_{\mathbb{X}_T} \right\|_{L^\infty}\\
		&+  \left\| \sum_{m=0}^{+\infty}\dfrac{m^2}{(1 +m )^{4}} t^{-\frac{3}{2}}\int_0^{\frac{t}{2}}\int_{B(x,\sqrt{t})}\left(\left(|u^1|+ |u^2| \right) |\bar{u}|+ \left(|\nabla Q^1|+ |\nabla Q^2| \right) |\nabla\bar{Q}| \right) dyd\tau \right\|_{L^\infty}\\
		\lesssim& (1+t)\sum_{l=0}^3\left(\| Q^1\|_{\mathbb{X}_T}^l+ \| Q^2\|_{\mathbb{X}_T}^l\right)\|\bar{Q}\|_{\mathbb{X}_T} + \sup_{x\in\R^3,0<t\leq T} \left(t^{-\frac{3}{2}}\int_0^{\frac{t}{2}}\int_{B(x,\sqrt{t})}  |\bar{u}|^2+ |\nabla \bar{Q}|^2 dyd\tau\right)^{\frac{1}{2}} \\
		&\times \sup_{x\in\R^3,0<t\leq T} \left(t^{-\frac{3}{2}}\int_0^{\frac{t}{2}}\int_{B(x,\sqrt{t})} |u^1|^2+ |u^2|^2+ |Q^1|^2+ |Q^2|^2dyd\tau\right)^{\frac{1}{2}} \\
		\lesssim&  (1+t)\sum_{l=0}^3\left(\| (Q^1,u^1)\|_{\mathbb{E}_T}^l+ \| (Q^2,u^2)\|_{\mathbb{E}_T}^l\right)\|(\bar{Q},\bar{u})\|_{\mathbb{E}_T},\\
		\mathcal{N}_2 
		\lesssim& t\sum_{l=0}^3\left(\| Q^1\|_{\mathbb{X}_T}^l+ \| Q^2\|_{\mathbb{X}_T}^l\right)\|\bar{Q}\|_{\mathbb{X}_T}+ t^{\frac{1}{2}}\left\|  \int_{\frac{t}{2}}^t\dfrac{d\tau}{\sqrt{t-\tau}^{3/2}}\left(\| Q^1\|_{\dot{C}^{1/2}}\|\nabla\bar{Q}\|_{L^\infty}+ \| \bar{Q}\|_{\dot{C}^{1/2}}\|\nabla Q^2\|_{L^\infty}\right) \right\|_{L^\infty} \\
		&+ t^{\frac{1}{2}}\left\|  \int_{\frac{t}{2}}^t\dfrac{1}{\sqrt{t-\tau}}\dfrac{\left(\|\left(Q^1,u^1\right)\|_{\mathbb{E}_T}+\|\left(Q^2,u^2\right)\|_{\mathbb{E}_T}\right)\|\left(\bar{Q},\bar{u}\right)\|_{\mathbb{E}_T}}{\tau}d\tau \right\|_{L^\infty} \\
		\lesssim& \left(\|\left(Q^1,u^1\right)\|_{\mathbb{E}_T}+\|\left(Q^2,u^2\right)\|_{\mathbb{E}_T}\right)\|\left(\bar{Q},\bar{u}\right)\|_{\mathbb{E}_T}\left|   \int_{\frac{t}{2}}^t\dfrac{\sqrt{t}}{\sqrt{t-\tau}^{3/2}\cdot\tau^{3/4}}d\tau+ \int_{\frac{t}{2}}^t\dfrac{\sqrt{t}}{\sqrt{t-\tau}\cdot\tau} d\tau\right| \\
		&+ t\sum_{l=0}^3\left(\| Q^1\|_{\mathbb{X}_T}^l+ \| Q^2\|_{\mathbb{X}_T}^l\right)\|\bar{Q}\|_{\mathbb{X}_T} \\
		\lesssim& (1+t)\sum_{l=0}^3\left(\| (Q^1,u^1)\|_{\mathbb{E}_T}^l+ \| (Q^2,u^2)\|_{\mathbb{E}_T}^l\right)\|(\bar{Q},\bar{u})\|_{\mathbb{E}_T},\\
		\mathcal{N}_3\lesssim &R^{-\frac{3}{2}}\int_0^R \int_{B(x_0,5\sqrt{R})}\left( \left(|u^1|+ |u^2| \right) |\bar{u}|+ \left(|\nabla Q^1|+ |\nabla Q^2| \right) |\nabla\bar{Q}| +\sum_{l=0}^{3}\left(|Q^1|^l + |Q^2|^l\right)|\bar{Q}| \right) dyd\tau\\
		& + R^{-3/2} \int_0^R  \int_{B(x_0,5\sqrt{R})}\frac{1}{\sqrt{R}}\nabla_y \chi\left(\frac{y}{\sqrt{R}}\right) \left( |\nabla\bar{Q}|| Q^1|+ |\nabla Q^2|| \bar{Q}|\right) dyd\tau\\
		\lesssim&R\sum_{l=0}^3\left(\| Q^1\|_{\mathbb{X}_T}^l+ \| Q^2\|_{\mathbb{X}_T}^l\right)\|\bar{Q}\|_{\mathbb{X}_T} + \left(R^{-\frac{3}{2}}\int_0^{R}\int_{B(x_0,5\sqrt{R})} |u^1|^2 + |u^2|^2+ |\nabla Q^1|^2 + |\nabla Q^2|^2dyd\tau\right)^{\frac{1}{2}}\\
		&\times \left(R^{-\frac{3}{2}}\int_0^{R}\int_{B(x_0,5\sqrt{R})} |\bar{u}|^2+ |\nabla \bar{Q}|^2dyd\tau\right)^{\frac{1}{2}}  + R^{-2} \int_0^R\int_{B(x_0,5\sqrt{R})}\frac{\left(\| Q^1\|_{\mathbb{X}_T}+  \| Q^2\|_{\mathbb{X}_T}\right)\|\bar{Q}\|_{\mathbb{X}_T}}{\sqrt{\tau}}dyd\tau\\
		\lesssim& (1+T)\sum_{l=0}^3\left(\| (Q^1,u^1)\|_{\mathbb{E}_T}^l+ \| (Q^2,u^2)\|_{\mathbb{E}_T}^l\right)\|(\bar{Q},\bar{u})\|_{\mathbb{E}_T},\\
		\mathcal{N}_4\lesssim & \left\| \int_0^{\frac{t}{2}} e^{(t-\tau)\Delta}  \left(  \left(|u^1|+ |u^2| \right) |\bar{u}|+ \left(|\nabla Q^1|+ |\nabla Q^2| \right) |\nabla\bar{Q}| \right)  d\tau \right\|_{L^\infty(Q(x_0,5\sqrt{R}))}\\
		& + \left\| \int_0^{\frac{t}{2}} \nabla e^{(t-\tau)\Delta}  \left( |\nabla\bar{Q}|| Q^1|+ |\nabla Q^2|| \bar{Q}|\right)  d\tau \right\|_{L^\infty(Q(x_0,5\sqrt{R}))} + t\sum_{l=0}^3\left(\| Q^1\|_{\mathbb{X}_T}^l+ \| Q^2\|_{\mathbb{X}_T}^l\right)\|\bar{Q}\|_{\mathbb{X}_T}\\
		\lesssim& \sup_{x\in\R^3,0<t\leq T} \left(t^{-\frac{3}{2}}\int_0^{\frac{t}{2}}\int_{B(x,\sqrt{t})} |u^1|^2+ |u^2|^2+ |\nabla Q^1|^2+ |\nabla Q^2|^2  dyd\tau\right)^{\frac{1}{2}}\\
		& \times \sup_{x\in\R^3,0<t\leq T} \left(t^{-\frac{3}{2}}\int_0^{\frac{t}{2}}\int_{B(x,\sqrt{t})} |\bar{u}|^2+ |\nabla \bar{Q}|^2dyd\tau\right)^{\frac{1}{2}}  + t\sum_{l=0}^3\left(\| Q^1\|_{\mathbb{X}_T}^l+ \| Q^2\|_{\mathbb{X}_T}^l\right)\|\bar{Q}\|_{\mathbb{X}_T}\\
		& + \left\| \int_0^{\frac{t}{2}} \dfrac{d\tau}{\sqrt{t-\tau}\sqrt{\tau}} \left(\| Q^1\|_{\mathbb{X}_T}+ \| Q^2\|_{\mathbb{X}_T}\right)\|\bar{Q}\|_{\mathbb{X}_T} \right\|_{L^\infty}\\
		\lesssim &(1+t)\sum_{l=0}^3\left(\| (Q^1,u^1)\|_{\mathbb{E}_T}^l+ \| (Q^2,u^2)\|_{\mathbb{E}_T}^l\right)\|(\bar{Q},\bar{u})\|_{\mathbb{E}_T},\\
		\mathcal{N}_5
		\lesssim & t\sum_{l=0}^3\left(\| Q^1\|_{\mathbb{X}_T}^l+ \| Q^2\|_{\mathbb{X}_T}^l\right)\|\bar{Q}\|_{\mathbb{X}_T}+ \left\|  \int_{\frac{t}{2}}^t\dfrac{d\tau}{\sqrt{t-\tau}\sqrt{\tau}}\left(\|\nabla Q^1\|_{\mathbb{X}_T}+\|Q^2\|_{\mathbb{X}_T}\right)\| \bar{Q} \|_{\mathbb{X}_T} \right\|_{L^\infty} \\
		&+ \left\|  \int_{\frac{t}{2}}^t \dfrac{\left(\|(Q^1,u^1) \|_{\mathbb{E}_T}+\|(Q^2,u^2)\|_{\mathbb{E}_T}\right)\|\left(\bar{Q},\bar{u}\right)\|_{\mathbb{E}_T}}{\tau}d\tau \right\|_{L^\infty} \\
		\lesssim& (1+t)\sum_{l=0}^3\left(\| (Q^1,u^1)\|_{\mathbb{E}_T}^l+ \| (Q^2,u^2)\|_{\mathbb{E}_T}^l\right)\|(\bar{Q},\bar{u})\|_{\mathbb{E}_T}.
	\end{align*}
To estimate $\mathcal{N}_6$, we distinguish two cases:

\noindent(i). When $0\leq \tau\leq t\leq (\sqrt{R}/2)^2$, we have
		\begin{align*}
			\mathcal{N}_6\lesssim& \sup_{x_0\in\R^3,R\leq T} R \left\| \int_0^t\int_{\R^3} \dfrac{(1-\chi(y/\sqrt{R}))\left(\left(|u^1|+ |u^2| \right) |\bar{u}|+ \left(|\nabla Q^1|+ |\nabla Q^2| \right) |\nabla\bar{Q}|  \right)}{(\sqrt{t-\tau}+|x-y|)^4} dy d\tau\right\|_{L^\infty({Q(x_0,\sqrt{R})})}^2\\
			&+ \sup_{x_0\in\R^3,R\leq T} R \left\| \int_0^t\int_{\R^3} \dfrac{1-\chi(y/\sqrt{R})}{(\sqrt{t-\tau}+|x-y|)^4} \left( \sum_{l=0}^{3}\left(|Q^1|^l + |Q^2|^l\right)|\bar{Q}| \right)dy d\tau\right\|_{L^\infty({Q(x_0,\sqrt{R})})}^2\\
			& + \sup_{x_0\in\R^3,R\leq T} R \left\| \int_0^t\int_{\R^3} \dfrac{1}{(\sqrt{t-\tau}+|x-y|)^5} \left( |\nabla\bar{Q}|| Q^1|+ |\nabla Q^2|| \bar{Q}|\right) dy d\tau\right\|_{L^\infty({Q(x_0,\sqrt{R})})}^2\\
			\lesssim& \sup_{x_0\in\R^3,R\leq T}  \left\| \int_0^{t}\sum_{m=1}^{+\infty}\int_{m\leq \frac{|x-y|}{\sqrt{R}}\leq (m+1)}\dfrac{\left(\left(|u^1|+ |u^2| \right) |\bar{u}|+ \left(|\nabla Q^1|+ |\nabla Q^2| \right) |\nabla\bar{Q}|  \right)}{m^4\sqrt{R}^3} dy d\tau\right\|_{L^\infty({Q(x_0,\sqrt{R})})}^2\\
			&+ \sup_{x_0\in\R^3,R\leq T} R \left\| \int_0^{t}\sum_{m=1}^{+\infty}\int_{m\sqrt{R}\leq |x-y|\leq (m+1)\sqrt{R}}\dfrac{1}{(m\sqrt{R})^4} \left( \sum_{l=0}^{3}\left(|Q^1|^l + |Q^2|^l\right)|\bar{Q}| \right)dy d\tau\right\|_{L^\infty({Q(x_0,\sqrt{R})})}^2\\
			& + \sup_{x_0\in\R^3,R\leq T} \left\|  \int_0^{t} \dfrac{1}{\sqrt{t-\tau}}\cdot\dfrac{\left(\| Q^1\|_{\mathbb{X}_T}+ \| Q^2\|_{\mathbb{X}_T}\right)\|\bar{Q} \|_{\mathbb{X}_T}}{\sqrt{\tau}}d\tau \right\|_{L^\infty(Q(x_0,\sqrt{R}))}^2\\
			\lesssim& \left((1+t)\sum_{l=0}^3\left(\| (Q^1,u^1)\|_{\mathbb{E}_T}^l+ \| (Q^2,u^2)\|_{\mathbb{E}_T}^l\right)\|(\bar{Q},\bar{u})\|_{\mathbb{E}_T}\right)^2 \\
			&+ \sup_{x_0\in\R^3,0<R\leq T} \left( \sup_{x\in B(x_0,\sqrt{R})} R^{-3/2} \int_0^{R} \int_{B(x,\sqrt{R})}|\bar{u}|^2+ |\nabla \bar{Q}|^2dyd\tau \right) \\
			&\times\sup_{x_0\in\R^3,0<R\leq T} \left( \sup_{x\in B(x_0,\sqrt{R})} R^{-3/2} \int_0^{R} \int_{B(x,\sqrt{R})}|u^1|^2+|u^2|^2+ |\nabla Q^1|^2 + |\nabla Q^2|^2 dyd\tau \right)\\
			\lesssim & \left((1+t)\sum_{l=0}^3\left(\| (Q^1,u^1)\|_{\mathbb{E}_T}^l+ \| (Q^2,u^2)\|_{\mathbb{E}_T}^l\right)\|(\bar{Q},\bar{u})\|_{\mathbb{E}_T}\right)^2.
		\end{align*}
\noindent(ii) When $(\sqrt{R}/2)^2\leq t\leq R$, we have
		\begin{align*}
			&\sup_{x_0\in\R^3,0<R\leq T} R \left\|\nabla\int_0^\frac{t}{2} e^{(t-\tau)\Delta} (1-\chi) \mathbb{N}_2(\bar{Q},\bar{u})d\tau\right\|_{L^\infty({Q(x_0,\sqrt{R})})}^2\\
			\lesssim& \sup_{x_0\in\R^3,0<R\leq T} R \left\| \int_0^\frac{t}{2}\int_{\R^3} \dfrac{\left(|u^1|+ |u^2| \right) |\bar{u}|+ \left(|\nabla Q^1|+ |\nabla Q^2| \right) |\nabla\bar{Q}|}{(\sqrt{t-\tau}+|x-y|)^4} dy d\tau\right\|_{L^\infty({Q(x_0,\sqrt{R})})}^2\\
			&+ \sup_{x_0\in\R^3,0<R\leq T} R \left\| \int_0^\frac{t}{2}\int_{\R^3} \dfrac{1}{(\sqrt{t-\tau}+|x-y|)^4} \left( \sum_{l=0}^{3}\left(|Q^1|^l + |Q^2|^l\right)|\bar{Q}| \right)dy d\tau\right\|_{L^\infty({Q(x_0,\sqrt{R})})}^2\\
			& + \sup_{x_0\in\R^3,0<R\leq T} R \left\| \int_0^\frac{t}{2}\int_{\R^3} \dfrac{1}{(\sqrt{t-\tau}+|x-y|)^5} \left(|\nabla\bar{Q}|| Q^1|+ |\nabla Q^2|| \bar{Q}|\right) dy d\tau\right\|_{L^\infty({Q(x_0,\sqrt{R})})}^2\\
			\lesssim& \sup_{x_0\in\R^3,R\leq T} \left\| t^{-\frac{3}{2}} \int_0^{\frac{t}{2}}\sum_{m=0}^{+\infty}\int_{m\leq \frac{|x-y|}{\sqrt{t}}\leq m+1}\dfrac{\left(|u^1|+ |u^2| \right) |\bar{u}|+ \left(|\nabla Q^1|+ |\nabla Q^2| \right) |\nabla\bar{Q}|}{(1+m)^4} dyd\tau \right\|_{L^\infty(Q(x_0,\sqrt{R}))}^2\\
			&+ \sup_{x_0\in\R^3,0<R\leq T} \left\| t^{-3/2} \int_0^{\frac{t}{2}}\sum_{m=0}^{+\infty}\int_{m\leq \frac{|x-y|}{\sqrt{t}}\leq m+1}\dfrac{1}{(1+m)^4} \left( \sum_{l=0}^{3}\left(|Q^1|^l + |Q^2|^l\right)|\bar{Q}| \right)dyd\tau \right\|_{L^\infty(Q(x_0,\sqrt{R}))}^2\\
			& + \sup_{x_0\in\R^3,0<R\leq T} \left(  \int_0^{\frac{t}{2}} \dfrac{\sqrt{R}}{\sqrt{t-\tau}}\cdot\dfrac{1}{\sqrt{t-\tau}\sqrt{\tau}}d\tau \left(\| Q^1\|_{\mathbb{X}_T}+ \| Q^2\|_{\mathbb{X}_T}\right)\|\bar{Q} \|_{\mathbb{X}_T} \right)^2\\
			\lesssim & \left((1+t)\sum_{l=0}^3\left(\| (Q^1,u^1)\|_{\mathbb{E}_T}^l+ \| (Q^2,u^2)\|_{\mathbb{E}_T}^l\right)\|(\bar{Q},\bar{u})\|_{\mathbb{E}_T}\right)^2
		\end{align*}
	and
		\begin{align*}
			&\sup_{x_0\in\R^3,0<R\leq T} R \left\|\nabla\int_\frac{t}{2}^t e^{(t-\tau)\Delta} (1-\chi) \mathbb{N}_2(\bar{Q},\bar{u})d\tau\right\|_{L^\infty({Q(x_0,\sqrt{R})})}^2\\
			\lesssim& \sup_{x_0\in\R^3,0<R\leq T} R \left\| \int_\frac{t}{2}^t\int_{\R^3} \dfrac{\left(|u^1|+ |u^2| \right) |\bar{u}|+ \left(|\nabla Q^1|+ |\nabla Q^2| \right) |\nabla\bar{Q}|}{(\sqrt{t-\tau}+|x-y|)^4} dy d\tau\right\|_{L^\infty({Q(x_0,\sqrt{R})})}^2\\
			&+ \sup_{x_0\in\R^3,0<R\leq T} R \left\| \int_\frac{t}{2}^t\int_{\R^3} \dfrac{1}{(\sqrt{t-\tau}+|x-y|)^4} \left( \sum_{l=0}^{3}\left(|Q^1|^l + |Q^2|^l\right)|\bar{Q}| \right)dy d\tau\right\|_{L^\infty({Q(x_0,\sqrt{R})})}^2\\
			& + \sup_{x_0\in\R^3,0<R\leq T} R \left\| \int_\frac{t}{2}^t\int_{\R^3} \dfrac{1}{(\sqrt{t-\tau}+|x-y|)^5} \left( |\nabla\bar{Q}|| Q^1|+ |\nabla Q^2|| \bar{Q}| \right) dy d\tau\right\|_{L^\infty({Q(x_0,\sqrt{R})})}^2\\
			\lesssim& \sup_{x_0\in\R^3,0<R\leq T} \left(  \int_{\frac{t}{2}}^t \dfrac{\sqrt{R}}{\sqrt{\tau}}\cdot\dfrac{1}{\sqrt{t-\tau}\sqrt{\tau}}d\tau\left(\| (Q^1,u^1)\|_{\mathbb{E}_T}+ \| (Q^2,u^2)\|_{\mathbb{E}_T}\right)\|(\bar{Q},\bar{u})\|_{\mathbb{E}_T} \right)^2\\
			& + \sup_{x_0\in\R^3,0<R\leq T} \left(  \int_{\frac{t}{2}}^t \dfrac{1}{\sqrt{t-\tau}^{3/2}}\cdot\dfrac{\sqrt{R}}{\sqrt{\tau}}\cdot\dfrac{\left(\| Q^1\|_{\mathbb{X}_T}+ \| Q^2\|_{\mathbb{X}_T}\right)\|\bar{Q}\|_{\mathbb{X}_T} }{\tau^{1/4}}d\tau \right)^2\\
			&+ \left(   t\sum_{l=0}^3\left(\| Q^1\|_{\mathbb{X}_T}^l+ \| Q^2\|_{\mathbb{X}_T}^l\right)\|\bar{Q}\|_{\mathbb{X}_T} \right)^2\\ 
			\lesssim & \left((1+t)\sum_{l=0}^3\left(\| (Q^1,u^1)\|_{\mathbb{E}_T}^l+ \| (Q^2,u^2)\|_{\mathbb{E}_T}^l\right)\|(\bar{Q},\bar{u})\|_{\mathbb{E}_T}\right)^2.
		\end{align*}
%	Hence
%	\begin{align}
%		J_{42}\lesssim (1+t)\sum_{l=0}^3\left(\| (Q^1,u^1)\|_{X_T}^l+ \| (Q^2,u^2)\|_{X_T}^l\right)\|(\bar{Q},\bar{u})\|_{X_T}
%	\end{align}
Substituting all these estimates into \eqref{eq3.35}, we then infer that
	\begin{equation}\label{eq3.36}
			\|\mathbb{S}_2(\bar{Q},\bar{u})\|_{\mathbb{Y}_T}
			\lesssim (1+T)\sum_{l=0}^3\left(\| (Q^1,u^1)\|_{\mathbb{E}_T}^l+ \| (Q^2,u^2)\|_{\mathbb{E}_T}^l\right)\|(\bar{Q},\bar{u})\|_{\mathbb{E}_T}.
	\end{equation}
	Next, we deal with $\|\mathbb{S}_1(\bar{Q},\bar{u})\|_{\mathbb{X}_T}$, which is roughly equivalent to estimate
	\begin{align}\label{eq3.37}
		\|\mathbb{S}_1(\bar{Q},\bar{u})\|_{\mathbb{X}_T}\lesssim&\sum_{k=0}^1t^{\frac{k}{2}}\|\nabla^k\mathbb{S}_1(\bar{Q},\bar{u})\|_{L^\infty}+ \sup_{x_0\in\R^3,0<R\leq T}\left(R^{-3/2}\int_0^R\int_{B(x_0,\sqrt{R})}|\nabla \mathbb{S}_1(\bar{Q},\bar{u})(y,t)|^2 dy dt\right)^{\frac{1}{2}}\nonumber\\
		\lesssim& \underbrace{\left\|\int_0^\frac{t}{2} e^{(t-\tau)\Delta}\mathbb{N}_1(\bar{Q},\bar{u})d\tau\right\|_{L^\infty}}_{\mathcal{J}_1} + \underbrace{\left\|\int_\frac{t}{2}^t  e^{(t-\tau)\Delta}\mathbb{N}_1(\bar{Q},\bar{u})d\tau\right\|_{L^\infty}}_{\mathcal{J}_2}\nonumber\\
		&+ \underbrace{t^{\frac{1}{2}}\left\|\nabla\int_0^\frac{t}{2}  e^{(t-\tau)\Delta}\mathbb{N}_1(\bar{Q},\bar{u})d\tau\right\|_{L^\infty}}_{\mathcal{J}_3} + \underbrace{t^{\frac{1}{2}}\left\|\nabla\int_\frac{t}{2}^t  e^{(t-\tau)\Delta}\mathbb{N}_1(\bar{Q},\bar{u})d\tau\right\|_{L^\infty}}_{\mathcal{J}_4}\nonumber\\
		& +\sup_{x_0\in\R^3,R\leq T}\Bigg( \underbrace{ R^{-\frac{3}{2}}\int_0^R \int_{\R^3}| \chi \mathbb{N}_1(\bar{Q},\bar{u})| dyd\tau }_{\mathcal{J}_5}\cdot\Bigg(\underbrace{\left\| \int_0^{\frac{t}{2}} e^{(t-\tau)\Delta}\chi \mathbb{N}_1(\bar{Q},\bar{u})  d\tau\right\|_{L^\infty(Q(x_0,5\sqrt{R}))}}_{\mathcal{J}_6} \nonumber\\
		& + \underbrace{\left\| \int_{\frac{t}{2}}^t e^{(t-\tau)\Delta}\chi \mathbb{N}_1(\bar{Q},\bar{u}) d\tau \right\|_{L^\infty(Q(x_0,5\sqrt{R}))}}_{\mathcal{J}_7}\Bigg)\Bigg)^{\frac{1}{2}}\\
		& + \left( \underbrace{\sup_{x_0\in\R^3,0<R\leq T} R \left\|\nabla\int_0^t e^{(t-\tau)\Delta} (1-\chi) \mathbb{N}_1(\bar{Q},\bar{u})d\tau\right\|_{L^\infty({Q(x_0,\sqrt{R})})}^2}_{\mathcal{J}_8} \right)^{\frac{1}{2}}.\nonumber 
	\end{align}
	Indeed, one have
	\begin{align*}
		\mathcal{J}_{1}\lesssim& \left\|  \int_0^{\frac{t}{2}}\int_{\R^3}\dfrac{e^{-\frac{(x-y)^2}{4(t-\tau)}}}{\sqrt{t-\tau}^3}\left(  |\bar{u}||\nabla Q^1|+  |u^2||\nabla\bar{Q}| +\sum_{l=0}^{2}\left(|Q^1|^l + |Q^2|^l\right)|\bar{Q}| \right)dyd\tau \right\|_{L^\infty}\\
		& + \left\|  \int_0^{\frac{t}{2}} \nabla e^{(t-\tau)\Delta} \left(|Q^1||\bar{u}|+ |\bar{Q}||u^2| \right)d\tau \right\|_{L^\infty}\\
		\lesssim &t\sum_{l=0}^2\left(\| Q^1\|_{\mathbb{X}_T}^l+ \| Q^2\|_{\mathbb{X}_T}^l\right)\|\bar{Q}\|_{\mathbb{X}_T} + \left\| \int_0^{\frac{t}{2}} \dfrac{d\tau}{\sqrt{t-\tau}\sqrt{\tau}} \left(\| Q^1\|_{\mathbb{X}_T}+ \| u^2\|_{\mathbb{Y}_T}\right)\|(\bar{Q},\bar{u})\|_{\mathbb{E}_T} \right\|_{L^\infty}\\
		& +\left\| \dfrac{1}{\sqrt{t}^3} \sum_{m=0}^{+\infty}\int_0^{\frac{t}{2}}\int_{m\leq \frac{|x-y|}{\sqrt{t-\tau}}\leq m+1}e^{-m^2}\left( |\bar{u}||\nabla Q^1|+  |u^2||\nabla\bar{Q}|\right)dyd\tau \right\|_{L^\infty}\\
		\lesssim & \sup_{x\in\R^3,0<t\leq T} \left(t^{-\frac{3}{2}}\int_0^{\frac{t}{2}}\int_{B(x,\sqrt{t})} |u^2|^2+ |\nabla Q^1|^2dyd\tau\right)^{\frac{1}{2}} \times \left(t^{-\frac{3}{2}}\int_0^{\frac{t}{2}}\int_{B(x,\sqrt{t})} |\bar{u}|^2+ |\nabla \bar{Q}|^2dyd\tau\right)^{\frac{1}{2}}\\
		& +(1+t)\sum_{l=0}^2\left(\| Q^1\|_{\mathbb{X}_T}^l+ \| (Q^2,u^2)\|_{\mathbb{E}_T}^l\right)\|(\bar{Q},\bar{u})\|_{\mathbb{E}_T} \\
		\lesssim &(1+t)\sum_{l=0}^2\left(\| (Q^1,u^1)\|_{\mathbb{E}_T}^l+ \| (Q^2,u^2)\|_{\mathbb{E}_T}^l\right)\|(\bar{Q},\bar{u})\|_{\mathbb{E}_T},\\
		\mathcal{J}_{2}\lesssim& t\sum_{l=0}^2\left(\| Q^1\|_{\mathbb{X}_T}^l+ \| Q^2\|_{\mathbb{X}_T}^l\right)\|\bar{Q}\|_{\mathbb{X}_T}+ \left\|  \int_{\frac{t}{2}}^t\dfrac{d\tau}{\sqrt{t-\tau}\sqrt{\tau}}\left(\| Q^1\|_{\mathbb{X}_T}+\|u^2\|_{\mathbb{Y}_T}\right)\|\left(\bar{Q},\bar{u}\right)\|_{\mathbb{E}_T} \right\|_{L^\infty} \\
		&+ \left\|  \int_{\frac{t}{2}}^t \dfrac{\left(\| Q^1\|_{\mathbb{X}_T}+\|u^2\|_{\mathbb{Y}_T}\right)\|\left(\bar{Q},\bar{u}\right)\|_{\mathbb{E}_T}}{\tau}d\tau \right\|_{L^\infty} \\
		\lesssim& (1+t)\sum_{l=0}^2\left(\| (Q^1,u^1)\|_{\mathbb{E}_T}^l+ \| (Q^2,u^2)\|_{\mathbb{E}_T}^l\right)\|(\bar{Q},\bar{u})\|_{\mathbb{E}_T},\\
		\mathcal{J}_{3}\lesssim&  t^{\frac{1}{2}}\left\|  \int_0^{\frac{t}{2}}\int_{\R^3}\dfrac{1}{\left(\sqrt{t-\tau} +|x-y| \right)^{4}}\left(  |\bar{u}||\nabla Q^1|+  |u^2||\nabla\bar{Q}| +\sum_{l=0}^{2}\left(|Q^1|^l + |Q^2|^l\right)|\bar{Q}| \right)dyd\tau \right\|_{L^\infty}\\
		&+t^{\frac{1}{2}}\left\|  \int_0^{\frac{t}{2}}\left|\int_{\R^3}\dfrac{1}{\left(\sqrt{t-\tau} +|x-y| \right)^{5}}dy\right| \left(|Q^1||\bar{u}|+ |\bar{Q}||u^2|\right) d\tau \right\|_{L^\infty}\\
		\lesssim&  t\sum_{l=0}^2\left(\| Q^1\|_{\mathbb{X}_T}^l+ \| Q^2\|_{\mathbb{X}_T}^l\right)\|\bar{Q}\|_{\mathbb{X}_T}+  t^{\frac{1}{2}}\left\|  \int_0^{\frac{t}{2}}\dfrac{1}{t-\tau}\dfrac{\left(\| Q^1\|_{\mathbb{X}_T}+\|u^2\|_{\mathbb{Y}_T}\right)\|\left(\bar{Q},\bar{u}\right)\|_{\mathbb{E}_T}}{\sqrt{\tau}}d\tau \right\|_{L^\infty}\\
		&+ \left\| t^{-\frac{3}{2}}\sum_{m=0}^{+\infty} \int_0^{\frac{t}{2}}\int_{m\leq \frac{|x-y|}{\sqrt{t-\tau}}\leq m+1}\dfrac{|\bar{u}||\nabla Q^1|+  |u^2||\nabla\bar{Q}|}{(1 +m )^{4}} dyd\tau \right\|_{L^\infty}\\
		\lesssim & \sup_{x\in\R^3,0<t\leq T} \left(t^{-\frac{3}{2}}\int_0^{\frac{t}{2}}\int_{B(x,\sqrt{t})} |u^2|^2+ |\nabla Q^1|^2dyd\tau\right)^{\frac{1}{2}} \times \left(t^{-\frac{3}{2}}\int_0^{\frac{t}{2}}\int_{B(x,\sqrt{t})} |\bar{u}|^2+ |\nabla \bar{Q}|^2dyd\tau\right)^{\frac{1}{2}}\\
		& + (1+t)\sum_{l=0}^2\left(\| Q^1\|_{\mathbb{X}_T}^l+ \| (Q^2,u^2)\|_{\mathbb{E}_T}^l\right)\|(\bar{Q},\bar{u})\|_{\mathbb{E}_T}\\
		\lesssim&  (1+t)\sum_{l=0}^2\left(\| (Q^1,u^1)\|_{\mathbb{E}_T}^l+ \| (Q^2,u^2)\|_{\mathbb{E}_T}^l\right)\|(\bar{Q},\bar{u})\|_{\mathbb{E}_T},\\
		\mathcal{J}_{4} 
		\lesssim& t\sum_{l=0}^2\left(\| Q^1\|_{\mathbb{X}_T}^l+ \| Q^2\|_{\mathbb{X}_T}^l\right)\|\bar{Q}\|_{\mathbb{X}_T}+ t^{\frac{1}{2}}\left\|  \int_{\frac{t}{2}}^t\dfrac{d\tau}{\sqrt{t-\tau}^{3/2}}\left(\| Q^1\|_{\dot{C}^{1/2}}\|\bar{u}\|_{L^\infty}+ \| \bar{Q}\|_{\dot{C}^{1/2}}\|u^2\|_{L^\infty}\right) \right\|_{L^\infty} \\
		&+ t^{\frac{1}{2}}\left\|  \int_{\frac{t}{2}}^t \dfrac{1}{\sqrt{t-\tau}}\dfrac{\left(\| Q^1\|_{\mathbb{X}_T}+\|u^2\|_{\mathbb{Y}_T}\right)\|\left(\bar{Q},\bar{u}\right)\|_{\mathbb{E}_T}}{\tau}d\tau \right\|_{L^\infty} \\
		\lesssim&  (1+t)\sum_{l=0}^2\left(\| Q^1\|_{\mathbb{X}_T}^l+ \| (Q^2,u^2)\|_{\mathbb{E}_T}^l\right)\|(\bar{Q},\bar{u})\|_{\mathbb{E}_T} + t^{\frac{1}{2}}\left\|  \int_{\frac{t}{2}}^t\dfrac{\left(\|Q^1\|_{\mathbb{X}_T}+\|u^2\|_{\mathbb{Y}_T}\right)\|\left(\bar{Q},\bar{u}\right)\|_{\mathbb{E}_T}}{\sqrt{t-\tau}^{3/2}\cdot \tau^{3/4}} d\tau \right\|_{L^\infty}\\
		\lesssim& (1+t)\sum_{l=0}^2\left(\| (Q^1,u^1)\|_{\mathbb{E}_T}^l+ \| (Q^2,u^2)\|_{\mathbb{E}_T}^l\right)\|(\bar{Q},\bar{u})\|_{\mathbb{E}_T},\\
		\mathcal{J}_5
		\lesssim &R^{-3/2}\int_0^R \int_{B(x_0,5\sqrt{R})}\left( |\bar{u}||\nabla Q^1|+  |u^2||\nabla\bar{Q}| +\sum_{l=0}^{2}\left(|Q^1|^l + |Q^2|^l\right)|\bar{Q}| \right) dyd\tau\\
		& + R^{-3/2} \int_0^R  \int_{B(x_0,5\sqrt{R})}\frac{1}{\sqrt{R}}\nabla_y \chi\left(\frac{y}{\sqrt{R}}\right) \left(|Q^1||\bar{u}|+ |\bar{Q}||u^2|\right) dyd\tau\\
		\lesssim& \left(R^{-\frac{3}{2}}\int_0^{R}\int_{B(x_0,5\sqrt{R})} |u^2|^2+ |\nabla Q^1|^2dyd\tau\right)^{\frac{1}{2}}\cdot \left(R^{-\frac{3}{2}}\int_0^{R}\int_{B(x_0,5\sqrt{R})} |\bar{u}|^2+ |\nabla \bar{Q}|^2dyd\tau\right)^{\frac{1}{2}}\\
		& + R\sum_{l=0}^2\left(\| Q^1\|_{\mathbb{X}_T}^l+ \| Q^2\|_{\mathbb{X}_T}^l\right)\|\bar{Q}\|_{\mathbb{X}_T} +  R^{-2}\int_0^R\int_{B(x_0,5\sqrt{R})}\frac{1}{\sqrt{\tau}}dyd\tau\left(\| Q^1\|_{\mathbb{X}_T}+ \| u^2\|_{\mathbb{Y}_T}\right)\|(\bar{Q},\bar{u})\|_{\mathbb{E}_T}\\
		\lesssim& (1+T)\sum_{l=0}^2\left(\| (Q^1,u^1)\|_{\mathbb{E}_T}^l+ \| (Q^2,u^2)\|_{\mathbb{E}_T}^l\right)\|(\bar{Q},\bar{u})\|_{\mathbb{E}_T},\\
		\mathcal{J}_6
		\lesssim& \left\| \int_0^{\frac{t}{2}} e^{(t-\tau)\Delta}  \left( |\bar{u}||\nabla Q^1|+  |u^2||\nabla\bar{Q}| +\sum_{l=0}^{2}\left(|Q^1|^l + |Q^2|^l\right)|\bar{Q}| \right)  d\tau \right\|_{L^\infty(Q(x_0,5\sqrt{R}))}\\
		& + \left\| \int_0^{\frac{t}{2}} \nabla e^{(t-\tau)\Delta}  \left(|Q^1||\bar{u}|+ |\bar{Q}||u^2|\right)  d\tau \right\|_{L^\infty(Q(x_0,5\sqrt{R}))}\\
		\lesssim& \sup_{x\in\R^3,0<t\leq T} \left(t^{-\frac{3}{2}}\int_0^{\frac{t}{2}}\int_{B(x,\sqrt{t})} |u^2|^2+ |\nabla Q^1|^2dyd\tau\right)^{\frac{1}{2}}\cdot \left(t^{-\frac{3}{2}}\int_0^{\frac{t}{2}}\int_{B(x,\sqrt{t})} |\bar{u}|^2+ |\nabla \bar{Q}|^2dyd\tau\right)^{\frac{1}{2}}\\
		& + t\sum_{l=0}^2\left(\| Q^1\|_{\mathbb{X}_T}^l+ \| Q^2\|_{\mathbb{X}_T}^l\right)\|\bar{Q}\|_{\mathbb{X}_T} + \left\| \int_0^{\frac{t}{2}} \dfrac{d\tau}{\sqrt{t-\tau}\sqrt{\tau}} \left(\| Q^1\|_{\mathbb{X}_T}+ \| u^2\|_{\mathbb{Y}_T}\right)\|(\bar{Q},\bar{u})\|_{\mathbb{E}_T} \right\|_{L^\infty}\\
		\lesssim &(1+t)\sum_{l=0}^2\left(\| (Q^1,u^1)\|_{\mathbb{E}_T}^l+ \| (Q^2,u^2)\|_{\mathbb{E}_T}^l\right)\|(\bar{Q},\bar{u})\|_{\mathbb{E}_T},\\
		\mathcal{J}_7
		\lesssim & t\sum_{l=0}^2\left(\| Q^1\|_{\mathbb{X}_T}^l+ \| Q^2\|_{\mathbb{X}_T}^l\right)\|\bar{Q}\|_{\mathbb{X}_T}+ \left\|  \int_{\frac{t}{2}}^t\dfrac{d\tau}{\sqrt{t-\tau}\sqrt{\tau}}\left(\|  Q^1\|_{\mathbb{X}_T}+\|u^2\|_{\mathbb{Y}_T}\right)\|\left(\bar{Q},\bar{u}\right)\|_{\mathbb{E}_T} \right\|_{L^\infty} \\
		&+ \left\|  \int_{\frac{t}{2}}^t \dfrac{\left(\| Q^1\|_{\mathbb{X}_T}+\|u^2\|_{\mathbb{Y}_T}\right)\|\left(\bar{Q},\bar{u}\right)\|_{\mathbb{E}_T}}{\tau}d\tau \right\|_{L^\infty} \\
		\lesssim& (1+t)\sum_{l=0}^2\left(\| (Q^1,u^1)\|_{\mathbb{E}_T}^l+ \| (Q^2,u^2)\|_{\mathbb{E}_T}^l\right)\|(\bar{Q},\bar{u})\|_{\mathbb{E}_T}.
	\end{align*}
	For the remainder term $\mathcal{J}_8$, we still estimate it in two cases:
	
	\noindent(i). When $0\leq \tau\leq t\leq (\sqrt{R}/2)^2$, we have
		\begin{align*}
			\mathcal{J}_8
			\lesssim& \sup_{x_0\in\R^3,0<R\leq T} R \left\| \int_0^t\int_{\R^3} \dfrac{1-\chi(y/\sqrt{R})}{(\sqrt{t-\tau}+|x-y|)^4} \left( |\bar{u}||\nabla Q^1|+  |u^2||\nabla\bar{Q}|  \right)dy d\tau\right\|_{L^\infty({Q(x_0,\sqrt{R})})}^2\\
			&+ \sup_{x_0\in\R^3,0<R\leq T} R \left\| \int_0^t\int_{\R^3} \dfrac{1-\chi(y/\sqrt{R})}{(\sqrt{t-\tau}+|x-y|)^4} \left( \sum_{l=0}^{2}\left(|Q^1|^l + |Q^2|^l\right)|\bar{Q}| \right)dy d\tau\right\|_{L^\infty({Q(x_0,\sqrt{R})})}^2\\
			& + \sup_{x_0\in\R^3,0<R\leq T} R \left\| \int_0^t\int_{\R^3} \dfrac{1}{(\sqrt{t-\tau}+|x-y|)^5} \left(|Q^1||\bar{u}|+ |\bar{Q}||u^2|\right) dy d\tau\right\|_{L^\infty({Q(x_0,\sqrt{R})})}^2\\
			\lesssim& \sup_{x_0\in\R^3,0<R\leq T} R \left\| \int_0^{t}\sum_{m=1}^{+\infty}\int_{m\leq \frac{|x-y|}{\sqrt{R}}\leq (m+1)}\dfrac{1}{(m\sqrt{R})^4} \left( |\bar{u}||\nabla Q^1|+  |u^2||\nabla\bar{Q}| \right)dy d\tau\right\|_{L^\infty({Q(x_0,\sqrt{R})})}^2\\
			&+ \sup_{x_0\in\R^3,0<R\leq T} R \left\| \int_0^{t}\sum_{m=1}^{+\infty}\int_{m \leq \frac{|x-y|}{\sqrt{R}}\leq (m+1)}\dfrac{1}{(m\sqrt{R})^4} \left( \sum_{l=0}^{2}\left(|Q^1|^l + |Q^2|^l\right)|\bar{Q}| \right)dy d\tau\right\|_{L^\infty({Q(x_0,\sqrt{R})})}^2\\
			& + \sup_{x_0\in\R^3,0<R\leq T} \left\|  \int_0^{t} \dfrac{1}{\sqrt{t-\tau}}\cdot\dfrac{\left(\| Q^1\|_{\mathbb{X}_T}+ \| u^2\|_{\mathbb{Y}_T}\right)\|(\bar{Q},\bar{u})\|_{\mathbb{E}_T}}{\sqrt{\tau}}d\tau \right\|_{L^\infty(Q(x_0,\sqrt{R}))}^2\\
			\lesssim& \left((1+t)\sum_{l=0}^2\left(\| Q^1\|_{\mathcal{X}_T}^l+ \| (Q^2,u^2)\|_{\mathbb{E}_T}^l\right)\|(\bar{Q},\bar{u})\|_{\mathbb{E}_T}\right)^2 \\
			&+ \sup_{x_0\in\R^3,0<R\leq T} \left( \sup_{x\in B(x_0,\sqrt{R})} R^{-3/2} \int_0^{R} \int_{B(x,\sqrt{R})}|\bar{u}|^2+ |\nabla \bar{Q}|^2dyd\tau \right) \\
			&\times\sup_{x_0\in\R^3,0<R\leq T} \left( \sup_{x\in B(x_0,\sqrt{R})} R^{-3/2} \int_0^{R} \int_{B(x,\sqrt{R})}|u^2|^2+ |\nabla Q^1|^2 dyd\tau \right)\\
			\lesssim & \left((1+t)\sum_{l=0}^2\left(\| (Q^1,u^1)\|_{\mathbb{E}_T}^l+ \| (Q^2,u^2)\|_{\mathbb{E}_T}^l\right)\|(\bar{Q},\bar{u})\|_{\mathbb{E}_T}\right)^2.
		\end{align*}
		\noindent(ii). When $(\sqrt{R}/2)^2\leq t\leq R$, we have
		\begin{align*}
			&\sup_{x_0\in\R^3,0<R\leq T} R \left\|\nabla\int_0^\frac{t}{2} e^{(t-\tau)\Delta} (1-\chi) \mathbb{N}_1(\bar{Q},\bar{u})d\tau\right\|_{L^\infty({Q(x_0,\sqrt{R})})}^2\\
			\lesssim& \sup_{x_0\in\R^3,0<R\leq T} R \left\| \int_0^\frac{t}{2}\int_{\R^3} \dfrac{1}{(\sqrt{t-\tau}+|x-y|)^4} \left( |\bar{u}||\nabla Q^1|+  |u^2||\nabla\bar{Q}|  \right)dy d\tau\right\|_{L^\infty({Q(x_0,\sqrt{R})})}^2\\
			&+ \sup_{x_0\in\R^3,0<R\leq T} R \left\| \int_0^\frac{t}{2}\int_{\R^3} \dfrac{1}{(\sqrt{t-\tau}+|x-y|)^4} \left( \sum_{l=0}^{2}\left(|Q^1|^l + |Q^2|^l\right)|\bar{Q}| \right)dy d\tau\right\|_{L^\infty({Q(x_0,\sqrt{R})})}^2\\
			& + \sup_{x_0\in\R^3,0<R\leq T} R \left\| \int_0^\frac{t}{2}\int_{\R^3} \dfrac{1}{(\sqrt{t-\tau}+|x-y|)^5} \left(|Q^1||\bar{u}|+ |\bar{Q}||u^2|\right) dy d\tau\right\|_{L^\infty({Q(x_0,\sqrt{R})})}^2\\
			\lesssim& \sup_{x_0\in\R^3,0<R\leq T} \left\| t^{-3/2} \int_0^{\frac{t}{2}}\sum_{m=0}^{+\infty}\int_{m\leq \frac{|x-y|}{\sqrt{t}}\leq m+1}\dfrac{1}{(1+m)^4}\left( |\bar{u}||\nabla Q^1|+  |u^2||\nabla\bar{Q}|  \right)dyd\tau \right\|_{L^\infty(Q(x_0,\sqrt{R}))}^2\\
			&+ \sup_{x_0\in\R^3,0<R\leq T} \left\| t^{-3/2} \int_0^{\frac{t}{2}}\sum_{m=0}^{+\infty}\int_{m\leq \frac{|x-y|}{\sqrt{t}}\leq m+1}\dfrac{1}{(1+m)^4} \left( \sum_{l=0}^{2}\left(|Q^1|^l + |Q^2|^l\right)|\bar{Q}| \right)dyd\tau \right\|_{L^\infty(Q(x_0,\sqrt{R}))}^2\\
			& + \sup_{x_0\in\R^3,0<R\leq T} \left(  \int_0^{\frac{t}{2}} \dfrac{\sqrt{R}}{\sqrt{t-\tau}}\cdot\dfrac{1}{\sqrt{t-\tau}\sqrt{\tau}}d\tau \left(\| Q^1\|_{\mathbb{X}_T}+ \| u^2\|_{\mathbb{Y}_T}\right)\|(\bar{Q},\bar{u})\|_{\mathbb{E}_T} \right)^2\\
			\lesssim & \left((1+t)\sum_{l=0}^2\left(\| (Q^1,u^1)\|_{\mathbb{E}_T}^l+ \| (Q^2,u^2)\|_{\mathbb{E}_T}^l\right)\|(\bar{Q},\bar{u})\|_{\mathbb{E}_T}\right)^2.
		\end{align*}
		and
		\begin{align*}
			&\sup_{x_0\in\R^3,0<R\leq T} R \left\|\nabla\int_\frac{t}{2}^t e^{(t-\tau)\Delta} (1-\chi) \mathbb{N}_1(\bar{Q},\bar{u})d\tau\right\|_{L^\infty({Q(x_0,\sqrt{R})})}^2\\
			\lesssim& \sup_{x_0\in\R^3,0<R\leq T} R \left\| \int_\frac{t}{2}^t\int_{\R^3} \dfrac{1}{(\sqrt{t-\tau}+|x-y|)^4} \left( |\bar{u}||\nabla Q^1|+  |u^2||\nabla\bar{Q}| \right)dy d\tau\right\|_{L^\infty({Q(x_0,\sqrt{R})})}^2\\
			&+ \sup_{x_0\in\R^3,0<R\leq T} R \left\| \int_\frac{t}{2}^t\int_{\R^3} \dfrac{1}{(\sqrt{t-\tau}+|x-y|)^4} \left( \sum_{l=0}^{2}\left(|Q^1|^l + |Q^2|^l\right)|\bar{Q}| \right)dy d\tau\right\|_{L^\infty({Q(x_0,\sqrt{R})})}^2\\
			& + \sup_{x_0\in\R^3,0<R\leq T} R \left\| \int_\frac{t}{2}^t\int_{\R^3} \dfrac{1}{(\sqrt{t-\tau}+|x-y|)^5} \left(|Q^1||\bar{u}|+ |\bar{Q}||u^2|\right) dy d\tau\right\|_{L^\infty({Q(x_0,\sqrt{R})})}^2\\
			\lesssim& \sup_{x_0\in\R^3,0<R\leq T} \left(  \int_{\frac{t}{2}}^t \dfrac{\sqrt{R}}{\sqrt{\tau}}\cdot\dfrac{1}{\sqrt{t-\tau}\sqrt{\tau}}d\tau\left(\| Q^1\|_{\mathbb{X}_T}+ \| u^2\|_{\mathbb{Y}_T}\right)\|(\bar{Q},\bar{u})\|_{\mathbb{E}_T} \right)^2\\
			& + \sup_{x_0\in\R^3,\sqrt{R}>0} \left(  \int_{\frac{t}{2}}^t \dfrac{1}{\sqrt{t-\tau}^{3/2}}\cdot\dfrac{\sqrt{R}}{\tau^{3/4}}d\tau \left(\| Q^1\|_{\mathbb{X}_T}+ \| u^2\|_{\mathbb{Y}_T}\right)\|(\bar{Q},\bar{u})\|_{\mathbb{E}_T} \right)^2\\
			&+ \left(   t\sum_{l=0}^2\left(\| Q^1\|_{\mathbb{X}_T}^l+ \|Q^2\|_{\mathbb{X}_T}^l\right)\|\bar{Q}\|_{\mathbb{X}_T} \right)^2 \\
			\lesssim & \left((1+t)\sum_{l=0}^2\left(\| (Q^1,u^1)\|_{\mathbb{E}_T}^l+ \| (Q^2,u^2)\|_{\mathbb{E}_T}^l\right)\|(\bar{Q},\bar{u})\|_{\mathbb{E}_T}\right)^2.
		\end{align*}
	We then conclude that
	\begin{equation}\label{eq3.38}
		\|\mathbb{S}_1(\bar{Q},\bar{u})\|_{\mathbb{X}_T}
			\lesssim (1+T)\sum_{l=0}^2\left(\| (Q^1,u^1)\|_{\mathbb{E}_T}^l+ \| (Q^2,u^2)\|_{\mathbb{E}_T}^l\right)\|(\bar{Q},\bar{u})\|_{\mathbb{E}_T}.
	\end{equation}
	Now, from \eqref{eq3.36} and \eqref{eq3.38}, one can see that
	\begin{equation}
		\begin{split}
			\|S(Q^1,u^1) - S(Q^2,u^2)\|_{\mathbb{E}_T} \lesssim& (1+T)\sum_{l=0}^3\left(\| (Q^1,u^1)\|_{\mathbb{E}_T}^l+ \| (Q^2,u^2)\|_{\mathbb{E}_T}^l\right)\|(\bar{Q},\bar{u})\|_{\mathbb{E}_T}\\
			\leq & C \epsilon_1 (1+T) \|(Q^1,u^1) - (Q^2,u^2)\|_{\mathbb{E}_T}\\
			\leq & \theta  \|(Q^1,u^1) - (Q^2,u^2)\|_{\mathbb{E}_T}
		\end{split}
	\end{equation}
	for some $\theta=\theta(\epsilon_1,T)\in (0,1)$, provided $\epsilon_1>0$ is chosen to be sufficiently small, where we have used
	\[
	\|Q^i\|_{\mathbb{X}_T} + \|u^i\|_{\mathbb{Y}_T}\lesssim \epsilon_1,\quad i=1,2.
	\]
	This completes the proof.
\end{proof}

%\begin{theorem}[local-wellposedness]
%	There exist $\epsilon_0>0$ such that if
%	\begin{equation}
%		\|Q_0\|_{L^\infty}+ \|u_0\|_{{\rm BMO}^{-1}}\leq \epsilon_0,
%	\end{equation}
%	then there exists $T_\ast>0$ such that system \eqref{eq-1.1} admits a unique solution $(Q,u)\in \mathbb{X}_{T_\ast}\times\mathbb{Y}_{T_\ast}$ with
%	\[
%	\|(Q,u)\|_{\mathbb{E}_{T\ast}}\leq C\epsilon_0.
%	\]
%\end{theorem}
\noindent\textbf{The proof of Theorem~\ref{thm1.1}}

%\begin{proof}
It follows directly from Lemma~\ref{lem3.1}, Lemma~\ref{lem3.2} and the Banach fixed-point theorem if $\epsilon_0$ is chosen to be sufficiently small.$\hfill\qedsymbol$

\section*{Acknowledgments}

%The authors would like to thank the referees for their constructive suggestions and comments, which help to improve the paper.
%The work of Q.~H.~Nguyen was supported by XXX.
%The work of F.~Yang
The author would like to thank Professor Quoc-Hung Nguyen for valuable discussions.
This work was supported in part by NSFC-12031012, NSFC-12250710674, NSFC-W2531006 and the Institute of Modern Analysis-A Frontier Research Center of Shanghai and the Jiangsu Innovation and Entrepreneurship Talent Program under Grant JSSCBS20250352 and the Scientific Research Foundation for Advanced Talent of Nantong University under Grant 135424639079.
%The author also appreciate Professor Quoc-Hung Nguyen for valuable discussions.
%The author Quoc-Hung Nguyen.

%

\end{document}